\documentclass[12pt,oneside]{amsart}
\usepackage{amsfonts}
\usepackage{amsmath}
\usepackage{amssymb}
\usepackage{mathrsfs}
\usepackage{graphicx}
\usepackage{subfigure}
\usepackage{hyperref}
\usepackage{microtype}
\usepackage{enumerate}
\usepackage{epsfig}
\usepackage{cancel}
\usepackage{color}
\usepackage{cite}
\usepackage{makecell}
\usepackage{multirow}
\usepackage{footmisc}
\usepackage{wrapfig}
\usepackage{array}
\usepackage{titletoc}
\usepackage{bm}
\usepackage{times}
\hypersetup{
colorlinks=true,
linkcolor=blue,
filecolor=blue,
urlcolor=blue,
citecolor=blue,}

\usepackage{geometry}
\usepackage{multirow}
\newtheorem{theorem}{Theorem}[section]
\newtheorem{lemma}[theorem]{Lemma}

\newtheorem{proposition}[theorem]{Proposition}

\numberwithin{equation}{section}

\theoremstyle{remark}

\allowdisplaybreaks[3]

\def\De{\Delta}

\def\na{\nabla}
\def\pat{\partial_t}

\def\lan{\langle}
\def\ran{\rangle}

\def\re{\operatorname{Re}}

\newcommand{\la}{\lambda}
\newcommand{\al}{\alpha}

\newcommand{\vp}{\varphi}

\newcommand{\R}{\mathbb{R}}

\newcommand{\Z}{\mathbb{Z}}
\newcommand{\T}{\mathbb{T}}

\newcommand{\f}{\forall}
\newcommand{\Om}{\Omega}
\newcommand{\om}{\omega}

\newcommand{\n}[1]{\Vert #1\Vert }
\newcommand{\bn}[1]{\big \Vert #1 \big \Vert }
\newcommand{\bbn}[1]{\Big\Vert #1 \Big \Vert }
\newcommand{\lr}[1]{\left\{ #1 \right\} }
\newcommand{\lrc}[1]{\left[ #1 \right] }
\newcommand{\lrs}[1]{\left( #1 \right) }
\newcommand{\lrv}[1]{\left| #1 \right| }
\newcommand{\lra}[1]{\langle #1 \rangle }
\newcommand{\blra}[1]{\big\langle #1\big\rangle }
\newcommand{\bblra}[1]{\Big\langle #1\Big\rangle }

\newcommand{\babs}[1]{\big | #1 \big|}
\newcommand{\bbabs}[1]{\Big | #1 \Big|}
\newcommand{\wt}[1]{\widetilde{#1}}

\newcommand{\pa}{\partial}

\begin{document}
\title[Asymptotic Stability for the 2-D Monotone Shear Flow]{Asymptotic stability threshold of the 2-D monotone shear flow with no-slip boundary condition}

\author[Z. Li]{Zhen Li}
\address{School of Mathematical Sciences, Key Laboratory of Mathematics and Complex Systems, Ministry of Education, Beijing Normal University, 100875 Beijing, China}
\email{lizhen@bnu.edu.cn}

\author[S. Shen]{Shunlin Shen}
\address{School of Mathematical Sciences, University of Science and Technology of China, Hefei, 230026, China}
\email{slshen@ustc.edu.cn}

\author[Z. Zhang]{Zhifei Zhang}
\address{School of Mathematical Sciences, Peking University, Beijing, 100871, China}

\email{zfzhang@math.pku.edu.cn}



\begin{abstract}
In this paper, we investigate the asymptotic stability threshold problem for the 2-D Navier-Stokes equations in a finite channel with no-slip boundary conditions, around monotone shear flow $(U(t,y),0)$.
We establish that the flow is asymptotically stable under perturbations satisfying
$\|u^{\mathrm{in}}\|_{H^2}\leq c\nu^{\frac12}$. To achieve the stability threshold $\nu^{\frac{1}{2}}$, the key ingredients of the proof include: sharp resolvent estimates for the vorticity based on weak-type resolvent bounds; weighted space-time estimates for the vorticity;  pointwise estimates for the velocity.
Furthermore, we handle the nonlinear term through a divergence formulation, which facilitates the sharp application of the aforementioned space-time estimates.\end{abstract}

\maketitle

\section{Introduction}\label{sec:Introduction}

Hydrodynamic stability, a cornerstone of fluid mechanics, traces its roots to Reynolds' seminal experiment, which revealed that the transition between laminar and turbulent flow states is governed by the Reynolds number $Re$, a dimensionless parameter quantifying the balance between inertial and viscous forces. In this context, the incompressible Navier-Stokes equations serve as the canonical mathematical framework for describing such viscous incompressible flows, and investigating their stability around shear flows in confined geometries (e.g., finite channels):
\begin{equation}\label{NS}
	\left\{
    \begin{aligned}
    &\pat v + (v\cdot\na)v -\nu\De v+ \na P =0, \\
    &\na\cdot v = 0.
    \end{aligned}
	\right.
	\end{equation}
Here, $v$ is the velocity field, $P$ is the pressure, and $\nu=Re^{-1}>0$ is the viscosity coefficient.

A key mathematical question in this field-first posed heuristically by Trefethen et al. \cite{Trefethen} and later formalized into a rigorous mathematical framework by Bedrossian, Germain, and Masmoudi \cite{Bedrossian-Germain-Masmoudi-4}, concerns the transition threshold for flow stability:{

\it  Given a norm $\n{\cdot}_X$, determine a $\beta=\beta(X)$ such that
\begin{align*}
&\|u^{\mathrm{in}}\|_{X}\ll \nu^{\beta}\Longrightarrow \text{ stability},\\
&\|u^{\mathrm{in}}\|_{X}\gg \nu^{\beta}\Longrightarrow \text{ instability}.
\end{align*}
}
\noindent This exponent $\beta$ quantifies the maximum amplitude of initial perturbations that the base flow can tolerate while remaining stable. Its scaling with $\nu$ directly reflects the interplay between viscous damping, inviscid effects, and boundary constraints, making it a benchmark for understanding transition mechanisms.

Over the past decade, significant progress has been made in resolving the transition threshold problem for prototypical shear flows, including plane Couette flow \cite{Bedrossian-Germain-Masmoudi-4,Bedrossian-Germain-Masmoudi-1,Bedrossian-Germain-Masmoudi-2,Chen-Wei-Zhang2,Wei-Zhang-1}, plane Poiseuille flow \cite{Chen-Ding-Lin-Zhang}, and Kolmogorov flow \cite{Li-Wei-Zhang1}. These investigations highlight that two key physical effects, inviscid damping and enhanced dissipation driven by the shear-flow-induced mixing mechanism, are instrumental in the early stage of the transition. For further details, see the review papers \cite{WZ-ICM, Bedrossian-Germain-Masmoudi-4}.

The two-dimensional incompressible Navier-Stokes equations do not develop finite-time singularities, yet their long-time dynamics at high Reynolds numbers are remarkably rich. Even near steady solutions, distinct dynamical behaviors can arise under different topological constraints. For instance, near the Couette flow, traveling wave solutions exist under low-regularity Sobolev perturbations \cite{LZ}, whereas the Couette flow is asymptotically stable under strong Gevrey-class perturbations due to the inviscid damping effect  \cite{BM}.  In this paper, we investigate the {\bf asymptotic stability threshold problem} for shear flows, that is, {\it we determine the size of initial perturbations (depending on various physical factors) under which solutions to the two-dimensional Navier-Stokes equations at high Reynolds numbers are asymptotically stable and satisfy both enhanced dissipation and inviscid damping.}

The asymptotic stability threshold problem for the Couette flow is now well-understood, with the stability threshold $\beta$ found to depend delicately on three key factors:
the regularity of the initial perturbation, the spatial domain, and the imposed boundary conditions. To contextualize our work, we summarize key existing results: \smallskip

 For the unbounded domain $\Om=\mathbb{T}\times \R$,
\begin{itemize}

\item if $X$ is  Gevrey class $2+$, then $\beta=0$ \cite{Bedrossian-Masmoudi-Vicol};

\item if $X$ is Sobolev space, then $\beta\le \frac13$ \cite{Masmoudi-Zhao, Wei-Zhang-2};

\item if $X$ is Gevrey class $\frac 1s$, $ s\in[0,\frac12]$, then $\beta\le \frac{1-2s}{3(1-s)}$ \cite{Li-Masmoudi-Zhao}.
\end{itemize}

For the bounded channel $\Om=\mathbb{T}\times [-1,1]$, the following results have been established under different physical boundary conditions:
\begin{itemize}
\item if $X$ is Sobolev space with no-slip BC, then $\beta\le \frac12$ \cite{Chen-Li-Wei-Zhang};

\item if $X$ is Sobolev space with Navier-slip BC, then $\beta\le \frac13$ \cite{Wei-Zhang-3};

\item if $X$ is  Gevrey class $2+$ with Navier-slip BC, then $\beta=0$ \cite{Bedrossian-He-Iyer-Wang1}.
\end{itemize}

For further developments on this topic, we refer the reader to \cite{An-He-Li, Arbon-Bedrossian, Bedrossian-He-Iyer-Wang, bgmz-cmp2025, cl-non2025, Chen-Li-Miao, Chen-Jia-Wei-Zhang,Gal,Zelati-Elgindi-Widmayer, del-sima2023, Jia, Li-Zhao, Masmoudi-Zhao,  Wei-Zhang-3, Wei-Zhang-2, lmz-cpam2024, mz-cpde2020, Wei-Zhang-Zhao-2} and the references therein.

While this body of work has focused primarily on the Couette flow, a natural and practically relevant extension is to ask: To what extent do the stability mechanisms (inviscid damping, enhanced dissipation) and threshold scalings identified for the Couette flow persist for more general, time-dependent monotone shear flows? The present paper is to extend the stability threshold analysis to a broad class of monotone shear profiles whose evolution is governed by a diffusion process.

Specifically, we consider a background shear flow $U(t,y)$ satisfying the one-dimensional heat equation with boundary conditions
\begin{equation}
\left\{
\begin{aligned}
&\pa_{t}U=\nu\pa_{y}^{2}U,\\
&U(0,y)=U^{\mathrm{in}}(y),\quad U(t,0)=U^{\mathrm{in}}(0),\quad U(t,1)=U^{\mathrm{in}}(1).
\end{aligned}
\right.
\end{equation}
The initial profile $U^{\mathrm{in}}(y)$ is a strictly monotone, and either concave or convex function:
\begin{equation}\label{equ:condition,shear,flow,M}
(\mathrm{M}):
\left\{
\begin{aligned}
 &U^{\mathrm{in}}\in H^4(0,1),\quad \quad \pa_{y}U^{\mathrm{in}}\geq c_0>0,\\
 &\text{$\pa_{y}^{2}U^{\mathrm{in}}\geq 0$, or $\pa_{y}^{2}U^{\mathrm{in}}\leq 0$},
\quad \pa_{y}^{2}U^{\mathrm{in}}(0)=\pa_{y}^{2}U^{\mathrm{in}}(1)=0.
\end{aligned}
\right.
\end{equation}

Let $u = v - (U(t,y), 0)$ denote the velocity perturbation. Under the no-slip boundary condition on $u$, we derive that
\begin{equation}\label{pertu}
\left\{
\begin{aligned}
&\pat u+u\cdot\na u-\nu \De u+(\pa_yU u^{(2)},0)+U(t,y)\pa_xu+\na p=0,\\
&\na\cdot u=0,\\
&u(t,x,0)=u(t,x,1)=0,\quad u(0,x,y)=u^{\mathrm{in}}(x,y).
\end{aligned}
\right.
\end{equation}

By introducing the vorticity $\om=\pa_y u^{(1)}-\pa_x u^{(2)}$ and the stream function $\psi$ defined via the relation $u=(\pa_{y}\psi,-\pa_x \psi)$,
the governing system can be reformulated in terms of the vorticity:
\begin{equation}\label{equ: omli11}
\left\{
\begin{aligned}
&\pat\om-\nu\De \om+u\cdot\na \om+U(t,y)\pa_x\om-\pa^2_{y}U(t,y)\psi=0,\\
&\De\psi=\om,\quad  u=(\pa_{y}\psi,-\pa_x \psi),\\
&\psi(t,x,0)=\psi(t,x, 1)=\pa_y\psi(t,x,0)=\pa_y \psi(t,x,1)=0,\\
&\om(0,x,y)=\om^{\mathrm{in}}(x,y).
\end{aligned}
\right.
\end{equation}

The main result of this paper is stated as follows.

\begin{theorem}\label{Th: tran thre}
Let $(\om,\psi)$ be the solution to \eqref{equ: omli11}. There exist positive constants $\nu_0$, $\epsilon_{0}$ and $c$, such that if the initial perturbation satisfies $\|u^{\mathrm{in}}\|_{H^2}\leq c\nu^{\frac12}$, $0<\nu\leq \nu_{0}$, then the solution $(\om,u)$ satisfies the global stability estimate
\begin{align*}
\sum_{k\in\Z} E_k \leq C\nu^{\frac12},
\end{align*}
where the stability norm is given by
\begin{equation*}
E_k=\left\{
\begin{aligned}
&\|\om_0\|_{L^\infty L^2}, \quad k=0,\\
&|k|\| e^{\epsilon_0 \nu^{\frac13}t} u_k\|_{L^2L^2}+\|e^{\epsilon_0 \nu^{\frac13}t} u_k\|_{L^\infty L^\infty}\\
&+\|e^{\epsilon_0 \nu^{\frac13}t} \sqrt{1-(2y-1)^2} \om_k\|_{L^\infty L^2}+\nu^{\frac14}|k|^{\frac12}\|e^{\epsilon_0 \nu^{\frac13}t} \om_k\|_{L^2L^2}, \quad k\neq 0,
\end{aligned}
\right.
\end{equation*}
with $f_k(t,y)=:\int_{\T} f(t,x,y) e^{-ikx}dx$. Here, $\n{f}_{L^p L^q}$ denotes the mixed norm $\n{f}_{L_t^p L_y^q}$.
\end{theorem}

Let us give some remarks on Theorem \ref{Th: tran thre}.

\begin{itemize}

\item
Theorem \ref{Th: tran thre} establishes the asymptotic stability of monotone shear flows under the no-slip boundary condition.
The asymptotic stability threshold $\beta=\frac{1}{2}$ is expected to be sharp, in view of its consistency with both the stability result \cite{Chen-Li-Wei-Zhang} and the instability result \cite{bgmz-cmp2025}
for the Couette flow.

\item
The present work attains this sharp threshold for monotone shear flows under no-slip boundary conditions, with critical reliance on two key estimates: the sharp inviscid damping estimate $|k|\n{u_{k}}_{L_{t,y}^{2}}$  and the enhanced dissipation estimate $\nu^{\frac{1}{4}}|k|^{\frac{1}{2}}\n{\om_{k}}_{L_{t,y}^{2}}$.

\item The $\n{u(t)}_{L_{x,y}^{\infty}}$ estimate incorporated into the stability norm implies the absolute stability of the monotone shear flows. This pointwise estimate is highly non-trivial, and its proof relies on the sharp $L_{t,y}^{2}$ space-time estimates.

\end{itemize}

The proof of Theorem \ref{Th: tran thre}, while built upon the framework introduced in \cite{Chen-Li-Wei-Zhang}, employs a distinct strategy based on a divergence formulation for the nonlinear terms.
This approach enables a sharp application of the key space-time estimates established in this work, and offers a new and effective pathway to tackle nonlinear stability problems,
with broad potential applicability to a variety of problems in different settings. We conclude this section by outlining the main ideas used to overcome the key difficulties.

\vspace{0.2cm}

\noindent\textbf{Difficulties and key ingredients.}

\vspace{0.1cm}

In the case of $\T\times \R$, both inviscid-damping and enhanced-dissipation estimates for the Couette flow can be elegantly derived through Fourier-multiplier techniques \cite{Bedrossian-Germain-Masmoudi-3,Bedrossian-Masmoudi-Vicol,Masmoudi-Zhao, Wei-Zhang-2}.
However, this approach is not directly applicable in the channel domain $\T\times [-1,1]$. Especially for the no-slip boundary condition, beyond difficulties already present in the Navier-slip case, an additional challenge arises due to boundary layer effects. A systematic framework to address these problems was developed in \cite{Chen-Li-Wei-Zhang}, which successfully established a nonlinear stability threshold of
$\beta \leq \frac{1}{2}$.
Following this route, the linear stability analysis for general monotone shear
 flows has been well established in \cite{Chen-Wei-Zhang1}, which serves as a key step towards the full nonlinear analysis of the transition threshold problem. In another direction, near-Couette flows have also been well-studied, with a sharp result (up to a logarithmic loss) recently established in \cite{Bedrossian-He-Iyer-Li-Wang}.

Compared with the Couette flow studied in \cite{Chen-Li-Wei-Zhang}, the analysis for general monotone shear profiles introduces substantial new difficulties. To overcome them, the nonlinear stability analysis requires new ingredients, which are stated below.

\begin{enumerate}[$(1)$]
\item Variable-coefficient problems.

In the context of linear stability analysis, the profile of the shear flows $U(t,y)$ plays a central role in determining the stability properties of the linearized system. For steady monotone shear flows $U(y)$, under an additional spectral assumption,
the linearized operator has been thoroughly studied in \cite{Chen-Wei-Zhang1}. However, in the time-dependent setting, the problem becomes more challenging due to the presence of a linearized operator with time-varying coefficients, which complicates the derivation of space-time estimates.
An effective strategy to address such variable-coefficient problems is the frozen-time approach. Nevertheless, this method faces additional difficulties. First, the control constant may vary per each time interval, leading to an accumulation of blow-up quantities. Second, the requisite spectral condition for the stability analysis must be satisfied uniformly in time.

Here, to deal with the two difficulties,
for monotone shear flows, two key observations facilitate the analysis. First, under the Neumann boundary conditions, a conservation property ensures a uniform lower bound on the velocity gradient
\begin{align}
0<c_{0}\leq \inf_{t\in[0,\infty)}\inf_{y\in[0,1]}\pa_{y}U.
\end{align}
Second, an instantaneous positivity property holds for the second derivative $\pa_{y}^{2}U$. Specifically,
\begin{align}
&\text{if $\pa_{y}^{2}U^{\mathrm{in}}\geq 0$, then $\pa_{y}^{2}U\equiv0$, or $\pa_{y}^{2}U(t,y)>0$, $\f t>0, y\in (0,1)$}, \\
&\text{if $\pa_{y}^{2}U^{\mathrm{in}}\leq 0$, then $\pa_{y}^{2}U\equiv0$, or $\pa_{y}^{2}U(t,y)<0$, $\f t>0, y\in (0,1)$}.
\end{align}
This property implies that even if embedding eigenvalue is present at the initial time, it is excluded for all positive time.

\item Sharp resolvent estimates with no-slip boundary condition.

The main difficulty in refining resolvent estimates for the vorticity is the nonlocal term $U''\pa_{x}\Delta^{-1}w$.
To beat it, we develop weak-type resolvent estimates to sharpen the coefficient bounds. As stated in Lemma \ref{lemma: weak resolvent estimate}, we prove two weak-type resolvent estimates according to the region the spectrum localized:
\begin{align*}
|\lan w_{Na},f\ran|
\lesssim &\nu^{-\frac14}|k|^{-\frac34}\lrv{f(j)}\n{F}_{H_{k}^{-1}}\lrs{|U(j)-\la| + \delta_{0}}^{-\frac34} \\
&+\nu^{-\frac12}|k|^{-\frac12}\n{F}_{H_{k}^{-1}} \lrs{|k|^{-1}\|\pa_y f\|_{L^\infty}+\|f\|_{L^\infty}},\notag
\end{align*}
and
\begin{align*}
|\lan w_{Na},f\ran|
\lesssim& |k|^{-1}\|F\|_{H_{k}^{-1}}\Big[ \nu^{-\frac14}|k||f(j)|(|k(U(j)-\la)|+1)^{-\frac34}+\Big\|\pa_{y}\lrs{f\frac{\chi_{\delta_{1}}}{U-\la}}\Big\|_{L^{2}}\\
&+(\nu^{-\frac13}|k|^{\frac13}+\nu^{-\frac12} |k|^{-1})\Big\|f\frac{\chi_{\delta_{1}}}{U-\la}\Big\|_{L^2} +\nu^{-\frac12}\n{(\pa_{y},k)f}_{L^{2}(E_{\delta_{1}})}\Big].\notag
\end{align*}

\item Weighted space-time $L^{\infty} L^{2}$ estimate for the vorticity and $L^{\infty}L^{\infty}$ estimate for the velocity.

Building on the sharp resolvent estimates for vorticity, we further establish sharp space-time $L^2L^2$ estimate for the vorticity. For the nonlinear stability analysis,
 another key point here is to prove weighted space-time $L^{\infty} L^{2}$ estimate for the vorticity and $L^{\infty}L^{\infty}$ estimate for the velocity, that is,
$$\|e^{\epsilon_0 \nu^{\frac13}t} \sqrt{1-(2y-1)^2} \om_k\|_{L^\infty L^2},\quad \|e^{\epsilon_0 \nu^{\frac13}t} u_k\|_{L^\infty L^\infty}.$$
Such estimates are already nontrivial even in the simpler case of the Couette flow $U(y)=y$, as seen in \cite{Chen-Li-Wei-Zhang}. Here, the main difficulty also arises from the nonlocal term $U''\pa_{x}\Delta^{-1}w$. By employing a different strategy and incorporating refined technical arguments, we are nevertheless able to establish these desired weighted space-time and pointwise estimates.

\item Nonlinear stability analysis through a divergence formulation.

Rather than employing the vorticity representation for the nonlinear terms, we adopt a divergence formulation given by
\begin{align*}
&\pa_t \om_k-\nu (\pa^2_y-k^2) \om_k+ikU(t,y)\om_k-ik\pa^2_yU(t,y)\psi_k
:=-\pa_y (f^{1,1}_k+f^{2,1}_k)-ik(f^{1,2}_k+f^{2,2}_k),
\end{align*}
where the nonlinear terms are given by the velocity components
\begin{align*}
&f^{1,1}_k=i\sum_{l\in\Z} u^{(1)}_l(t,y)(k-l)u^{(1)}_{k-l}(t,y), \quad f^{1,2}_k=i\sum_{l\in\Z} u^{(1)}_l(t,y)(k-l)u^{(2)}_{k-l}(t,y),\\
&f^{2,1}_k=\sum_{l\in \Z} u^{(2)}_l(t,y) \pa_y u^{(1)}_{k-l}(t,y), \qquad\quad f^{2,2}_k=\sum_{l\in \Z} u^{(2)}_l(t,y) \pa_y u^{(2)}_{k-l}(t,y).
\end{align*}
The divergence formulation reduces the dependence on vorticity estimates and provides a new route for handling the nonlinear stability problem. This has proven effective in the study of symmetric shear flows, as demonstrated in \cite{Chen-Li-Shen-Zhang}.
It is particularly well-suited for the case of monotone shear flows,
allowing for a sharp application of weighted space-time $L^{\infty}L^{2}$ estimate and $L^\infty L^\infty$ estimate we obtained. Also,
it provides an alternative proof to achieve the stability threshold for the 2-D Couette flow as in \cite{Chen-Li-Wei-Zhang}.
\end{enumerate}

\medskip

\noindent{\bf Notations.}
Throughout this paper, $C$ denotes a general constant independent of $\nu, k, \la$, and it may vary from line to line. The notation $A\lesssim B$ means $A\leq CB$.
In the following, we will omit the subscript $k$ for $\om, w, \psi, \phi, F$ while still keep the dependence on $k$ in the actual estimates.

\section{Resolvent estimates of the linearized operator}\label{sec: reso esti}

This section is devoted to resolvent estimates for steady shear flows, which are central to the frozen-coefficient method. Specifically, we consider the Orr-Sommerfeld equation with the no-slip boundary condition
\begin{equation}\label{equ:non-slip,OS}
\left\{
\begin{aligned}
&-\nu(\pa^2_y-k^2)w +ik(V-\la)w -ik V''\phi+o(\nu,k)w=F ,\\
&\phi=(\pa^2_y-k^2)^{-1}w, \quad \phi(0)=\phi(1)=\pa_{y}\phi(0)=\pa_{y}\phi(1)=0,\\
\end{aligned}
\right.
\end{equation}
where $\la\in \R$, $|o(\nu,k)|\ll (\nu k^{2})^{\frac{1}{3}}$,
and the subscript $k$ in $w_{k}$ is omitted for brevity.
Here, the steady flow $V(y)$ satisfies
\begin{align}\label{cond V}
V(y)\in H^4(0,1), \quad \inf_{y\in(0,1)}\pa_{y}V\geq c_{0}>0,\,\,\, V'' >0 \,\,\text{or}\,\, V''<0.
\end{align}
The convexity condition in \eqref{cond V} guarantees that the Rayleigh operator
$$R_{k}=:(\pa^2_y-k^2)^{-1}(V(y)(\pa^2_y-k^2)-V''(y))$$
has no embedding eigenvalues or eigenvalues.
This spectral property is crucial for the linear stability analysis carried out in \cite{Chen-Wei-Zhang1}. See also \cite{Wei-Zhang-Zhao} for more details.

As a preliminary step, we decompose the solution to \eqref{equ:non-slip,OS} into
\begin{align}\label{equ:decomposition,wna,w1,w2}
w=w_{Na}+c_{1}w_{1}+c_{2}w_{2},
\end{align}
where $w_{Na}$ solves the inhomogeneous OS equation with the Navier-slip boundary condition
\begin{equation}\label{equ:Navier-slip,OS}
\left\{
\begin{aligned}
&-\nu(\pa^2_y-k^2)w_{Na} +ik(V-\la)w_{Na} -ik V''\phi_{Na}+o(\nu,k)w_{Na} =F ,\\
&\phi_{Na}=(\pa^2_y-k^2)^{-1}w_{Na}, \quad \phi_{Na}(0)=\phi_{Na}(1)=w_{Na}(0)=w_{Na}(1)=0,\\
\end{aligned}
\right.
\end{equation}
and $w_1$, $w_2$ solve the following homogeneous OS equations
\begin{equation}\label{equ: psi1}
\left\{
\begin{aligned}
&-\nu(\pa^2_y-k^2)w_{1}+i k(V-\la)w_{1}-ik V''\phi_{1}+o(\nu,k)w_{1}=0,\\
&\phi_{1}=(\pa^2_y-k^2)^{-1}w_{1},\quad \phi_{1}(0)=\phi_{1}(1)=0,\quad \pa_y\phi_{1}(0)=0, \quad \pa_y\phi_{1}( 1)=1,
\end{aligned}
\right.
\end{equation}
and
\begin{equation}\label{equ: psi2}
\left\{
\begin{aligned}
&-\nu(\pa^2_y-k^2)w_{2}+i k(V-\la)w_{2}-ik V''\phi_{2}+o(\nu,k)w_{2}=0,\\
&\phi_{2}=(\pa^2_y-k^2)^{-1}w_{2},\quad \phi_{2}(0)=\phi_{2}(1)=0,\quad \pa_y\phi_{2}(0)=1,\quad \pa_y\phi_{2}(1)=0.
\end{aligned}
\right.
\end{equation}
Using the boundary condition in \eqref{equ: psi1} and \eqref{equ: psi2}, the coefficients are given by
\begin{equation}\label{def: c1c2}
\begin{aligned}
c_{1}=\int^1_{0}\frac{\sinh  k(1-y)}{\sinh k}w_{Na}(y)dy,\qquad c_{2}=\int^1_{0}\frac{\sinh (ky)}{\sinh k}w_{Na}(y)dy.
\end{aligned}
\end{equation}

\subsection{Resolvent estimates with Naiver-slip boundary condition}
Resolvent estimates for the velocity $u=(\pa_{y},-ik)\phi$, boundary layer correctors estimates, and the coefficient estimates, are well-established in the literature \cite{Chen-Li-Wei-Zhang,Chen-Wei-Zhang1}.
For clarity and completeness, we summarize in Lemma \ref{pro: reso esti FL2} the key results from \cite[Proposition 3.13, Lemmas 4.2 and 4.4, Proposition 5.1]{Chen-Wei-Zhang1} and \cite[Lemma 5.2]{Chen-Li-Wei-Zhang}.
\begin{lemma}\label{pro: reso esti FL2}
Let $(w,\phi)$ be the solution to \eqref{equ:non-slip,OS}. Then
there exist constants $\nu_0>0$ and $\epsilon_1>0$ such that, for $\nu\leq \nu_0$, $\nu k^2\leq \|V'\|_{L^\infty}$ and $|o(\nu,k)|\leq \epsilon_{1}(\nu k^{2})^{\frac{1}{3}}$, the following estimates hold.
  \begin{enumerate}[$(1)$]
  \item Resolvent estimates with Navier-slip boundary condition:
 \begin{align}
 \nu^{\frac{1}{6}}|k|^{\frac{4}{3}}\|u_{Na}\|_{L^2}+\nu^{\frac{1}{3}} |k|^{\frac{2}{3}}\|w_{Na}\|_{L^2}+\nu^{\frac23}|k|^{\frac13}\|(\pa_y,k)w_{Na}\|_{L^2}\lesssim&  \|F\|_{L^2},\label{esti: wL2}\\
\nu^{\frac{1}{6}}|k|^{\frac43}\|w_{Na}\|_{L^2}+\nu^{\frac{1}{12}}|k|^{\frac{5}{3}}\|w_{Na}\|_{L^1}+ |k|^2\|(V(y)-\la)w_{Na}\|_{L^2} \lesssim & \|F\|_{H_{k}^{1}}, \label{esti: wH1}\\
\nu^{\frac{1}{2}}  |k|\left\|u_{Na}\right\|_{L^2}+\nu^{\frac{2}{3}} |k|^{\frac{1}{3}}\|w_{Na}\|_{L^2}+\nu\|(\pa_y,k) w_{Na}\|_{L^2}
 \lesssim & \|F\|_{H_{k}^{-1}}, \label{esti: wH-1}
\end{align}
where the velocity $u_{Na}=(\pa_y\phi_{Na},-ik\phi_{Na})$, and the norms are given by
\begin{align*}
\|F\|_{H_{k}^{1}}=\|(\pa_y,k) F\|_{L^2},\quad \|F\|_{H_{k}^{-1}}=\sup \{|\lan F,\phi\ran|: \phi\in H^1_0(0,1), \, \|\phi\|_{H_{k}^{1}}=1\}.
\end{align*}
\item Estimates for the boundary layer correctors:
\begin{align}
&\|w_1\|_{L^2}\lesssim \nu^{-\frac14}(1+|k(V(0)-\la)|)^{\frac14},\label{w1L2}\\
&\|w_2\|_{L^2}\lesssim \nu^{-\frac14}(1+|k(V(1)-\la)|)^{\frac14},\label{w2L2}\\
&\|\rho^{\frac12 }_kw_1\|_{L^2}+\|\rho^{\frac12}_k w_2\|_{L^2}\lesssim L^{\frac12}, \label{weightL2}
\end{align}
where $L=\nu^{-\frac13}|k|^{\frac13}$ and the weight function is given by
\begin{equation}\label{equ:def,rho,k}
\rho_{k}(y)=\left\{
\begin{aligned}
&Ly, \qquad\qquad\, y\in [0,L^{-1}], \\
&1, \qquad\qquad\quad y\in [L^{-1}, 1-L^{-1}],\\
&L(1-y), \qquad y\in [1-L^{-1}, 1].
\end{aligned}
\right.
\end{equation}
\item Estimates for the coefficients:
\begin{align}
|c_1|+|c_2|\lesssim& \nu^{-\frac{1}{4}}|k|^{-1}\|F\|_{L^2},\label{c1c2L2}\\
|c_1|+|c_2|\lesssim& \nu^{-\frac{1}{12}}|k|^{-\frac53}\|F\|_{H^{1}_{k}},\label{c1c2H1}\\
|c_1|+|c_2|\lesssim& \nu^{-\frac{7}{12}}|k|^{-\frac23}\|F\|_{H^{-1}_{k}}.\label{c1c2H-1}
\end{align}
\end{enumerate}
\end{lemma}
Notably, the coefficient estimates for $F\in H_{k}^{-1}$ in hand, combined with the boundary layer correctors estimates, are insufficient to yield results that align with the result in the simple case of Couette flow.
This is mainly due to the presence of the extra nonlocal term $V''\pa_{x}\Delta^{-1}\phi$.
In contrast, this nonlocal term vanishes for Couette flow where $V(y)=y$.
To bridge this gap, the next subsection is devoted to deriving refined estimates for the coefficients $c_{1}$ and $c_{2}$, thereby establishing resolvent estimates for the vorticity, that match the known optimal results in \cite{Chen-Li-Wei-Zhang}.

\subsection{Resolvent estimates with no-slip boundary condition}\label{sec: Non slip}
 In this subsection, we focus on deriving such sharp resolvent estimates for the vorticity, which are the main results presented below.
 \begin{proposition}\label{lemma:non-slip boundary,resolvent}
Let $(w,\phi)$ be the solution to \eqref{equ:non-slip,OS}. Then
there exist constants $\nu_0>0$ and $\epsilon_1>0$ such that, for $\nu\leq \nu_0$, $\nu k^2\leq \|V'\|_{L^\infty}$ and $|o(\nu,k)|\leq \epsilon_{1}(\nu k^{2})^{\frac{1}{3}}$, we have
\begin{align}
\nu^{\frac14}|k|^{\frac12}\|w\|_{L^2}+\nu^{\frac16}|k|^{\frac13}\|\rho^{\frac12}_k w\|_{L^2}\lesssim & \nu^{-\frac{1}{6}}|k|^{-\frac13}\|F\|_{L^2}, \label{esti: u,wL2FL^2}\\
\nu^{\frac14}|k|^{\frac12}\|w\|_{L^2}+\nu^{\frac16}|k|^{\frac13}\|\rho^{\frac12}_k w\|_{L^2} \lesssim & \nu^{-\frac{1}{12}}|k|^{-\frac76}\|F\|_{H_{k}^{1}}, \label{esti: uwL2FH1}\\
\nu^{\frac14}|k|^{\frac12}\|w\|_{L^2}+\nu^{\frac16}|k|^{\frac13}\|\rho^{\frac12}_k w\|_{L^2}\lesssim & \nu^{-\frac12}\|F\|_{H_{k}^{-1}}. \label{esti: u,wL2FH-1}
\end{align}
\end{proposition}

The proof of Proposition \ref{lemma:non-slip boundary,resolvent} crucially relies on the refined estimates for the coefficients. Prior to that, we first derive weak-type resolvent estimates in Lemma \ref{lemma: weak resolvent estimate}, which play a key role in sharpening the coefficient estimates presented later in Lemma \ref{lemma: c1,c2 FL2}.

A key aspect of weak-type resolvent estimates is the analysis in the region where $V(y)\sim \la$. To localize the behavior in this region, we introduce a $C^{1}$-smooth cutoff function $\chi_{\delta}(y)$:
\begin{equation}\label{fun: chi1}
\chi_{\delta}(y)=\left\{
\begin{aligned}
&\quad 1, \qquad  &y\in[0,1]\setminus E_{\delta},\\
&\frac{2(V(y)-\la)^2 }{\delta^2}-\frac{(V(y)-\la)^{4}}{\delta^4}, \quad &y\in E_{\delta},
\end{aligned}
\right.
\end{equation}
where $E_{\delta}:=\lr{y\in [0,1]:V(y)\in(\la-\delta,\la+\delta)}$. This cutoff function satisfies the following properties:
\begin{align}\label{esti: chi}
\frac{\chi_{\delta}}{V-\la}\in C^{1},\quad \bbn{\frac{\chi_{\delta}}{V-\la}}_{L^{2}}\lesssim \delta^{-\frac{1}{2}},\quad
\bbn{\frac{\chi_{\delta}}{V-\la}}_{L^{\infty}}\lesssim \delta^{-1},\quad \bbn{\pa_{y}\lrs{\frac{\chi_{\delta}}{V-\la}}}_{L^{2}}\lesssim \delta^{-\frac{3}{2}}.
\end{align}
 We also define the complementary cutoff function $\chi_{\delta}^{c}(y) = 1 - \chi_{\delta}(y)$.

\begin{lemma}[Weak-type resolvent estimates]\label{lemma: weak resolvent estimate}
Let $(w_{Na},\phi_{Na})$ be the solution to \eqref{equ:Navier-slip,OS}, and $f\in H^{1}(0,1)$ satisfy $f(1-j)=0$ for either $j=0$ or $j=1$.
Define
\begin{equation*}
A_{F}=\left\{
\begin{aligned}
&\nu^{-\frac16}|k|^{-\frac56}\|F\|_{L^2},\quad \text{if}\quad  F\in L^2,\\
&\nu^{-\frac12}|k|^{-\frac12}\|F\|_{H^{-1}_{k}},\quad \text{if}\quad F\in H^{-1}_{k}.
\end{aligned}
\right.
\end{equation*}
Then there exist constants $\nu_0>0$ and $\epsilon_1>0$ such that, for $\nu\leq \nu_0$, $\nu k^2\leq \|V'\|_{L^\infty}$ and $|o(\nu,k)|\leq \epsilon_{1}(\nu k^{2})^{\frac{1}{3}}$, the following estimates hold.
\begin{enumerate}[$(1)$]
\item If $0\leq |V(j)-\la|\leq 2|k|^{-1}$, for $\delta_{0}=\nu^{\frac13}|k|^{-\frac13}$, we have
\begin{align}\label{weak esti 1}
|\lan w_{Na},f\ran|
\lesssim &\lrs{\nu^{\frac14}|k|^{-\frac14}\lrv{f(j)}\lrs{|V(j)-\la| + \delta_{0}}^{-\frac34} +\lrs{|k|^{-1}\|\pa_y f\|_{L^\infty}+\|f\|_{L^\infty}}}A_{F}.
\end{align}
\item If $|V(j)-\la|\geq 2|k|^{-1}$, for $\delta_{1}=|k|^{-1}$, we have
\begin{align}\label{weak esti 2}
|\lan w_{Na},f\ran|
\lesssim& \Big[ \nu^{\frac14}|k|^{\frac12}|f(j)|(|k(V(j)-\la)|+1)^{-\frac34}+\nu^{\frac12}|k|^{-\frac12}\Big\|\pa_{y}\lrs{f\frac{\chi_{\delta_{1}}}{V-\la}} \Big\|_{L^{2}}\\
&+(\nu^{\frac16}|k|^{-\frac16}+|k|^{-\frac32})\Big\|f\frac{\chi_{\delta_{1}}}{V-\la}\Big\|_{L^2} +|k|^{-\frac12}\n{(\pa_{y},k)f}_{L^{2}(E_{\delta_{1}})}\Big]A_{F}.\notag
\end{align}
\end{enumerate}
\end{lemma}
\begin{proof}
The estimates for the case $F\in L^2$ are established following the same line as those for $F\in H^{-1}_k$, merely replacing the resolvent estimate \eqref{esti: wH-1} by \eqref{esti: wL2} as required in the subsequent proof. We therefore present only the proofs corresponding to the case $F\in H^{-1}_k$.

 For simplicity, we deal with the case $j=1$, since the case $j=0$  follows in a similar manner.
First, we rewrite
\begin{align}\label{inner wna,f}
\lra{w_{Na},f}=\blra{(V-\la)w_{Na},\frac{f\chi_{\delta}}{V-\la}}+\lra{w_{Na},f\chi_{\delta}^{c}}.
\end{align}
We then analyze $\lra{(V-\la)w_{Na},\vp}$. Testing \eqref{equ:Navier-slip,OS} with $\vp\in H_{0}^{1}(0,1)$ yields
\begin{align*}
\lra{-\nu(\pa^2_y-k^2)w_{Na} +ik(V-\la)w_{Na} -ik V''\phi_{Na}+o(\nu,k)w_{Na},\vp} =\lra{F,\vp}.
\end{align*}
Using the resolvent estimates \eqref{esti: wH-1} and the condition $\nu k^{2}\lesssim 1$, we derive
\begin{align*}
&|\lra{(V-\la)w_{Na}-V'' \phi_{Na},\vp}|\\
&\lesssim |k|^{-1}\lrs{\n{F}_{H_{k}^{-1}}\n{\vp}_{H_{k}^{1}}+\nu\n{w_{Na}'}_{L^{2}}\n{\vp'}_{L^{2}}
+(\nu k^{2}+\epsilon_{1}\nu^{\frac13}|k|^{\frac23})\n{w_{Na}}_{L^{2}}\n{\vp}_{L^{2}}}\notag\\
&\lesssim |k|^{-1}\n{F}_{H_{k}^{-1}} \lrs{\|\varphi\|_{H^{1}_{k}}+\nu^{-\frac13}|k|^{\frac13}\|\varphi\|_{L^{2}}},\notag
\end{align*}
which implies
\begin{equation}
\begin{aligned}\label{equ:U-la,w,vp}
|\lra{(V-\la)w_{Na},\vp}|
\lesssim |k|^{-1}\n{F}_{H_{k}^{-1}} \lrs{\|\varphi\|_{H^{1}_{k}}+\nu^{-\frac13}|k|^{\frac13}\|\varphi\|_{L^{2}}}+|\lra{V'' \phi_{Na},\vp}|.
\end{aligned}
\end{equation}

For $\varphi\in H^1(0,1)$ with $\varphi(0)=0$, we construct an auxiliary function $\vp_1 \in H_{0}^{1}(0,1)$ by
\begin{align}\label{equ:vp1,vp,decomposition}
\vp_{1}(y)=\vp(y)-\vp(1)\eta_{\delta_{*}}(y),
\end{align}
where the $C^{1}$-smooth function $\eta_{\delta_{*}}(y)$ satisfies
\begin{equation}\label{equ:eta,delta,star}
\left\{
\begin{aligned}
&0\leq \eta_{\delta_{*}}(y)\leq 1,\quad  \eta_{\delta_{*}}(1)=1,\quad \text{supp}\eta_{\delta_{*}} \subset [1-\delta_*, 1],\\
		&\n{\eta_{\delta_{*}}}_{L^\infty} = 1, \quad \n{\eta_{\delta_{*}}}_{L^2} \leq\delta_*^{\frac{1}{2}},\quad
\n{\eta_{\delta_{*}}}_{L^{1}}\leq \delta_{*},\quad		\n{\eta_{\delta_{*}}'}_{L^2} \leq \delta_*^{-\frac12}, \\
		&\n{(V-\la)\eta_{\delta_{*}}}_{L^\infty} \lesssim |V(1)-\la| + \delta_*.
	\end{aligned}
\right.
\end{equation}
It follows from \eqref{esti: wH-1} that
\begin{align*}
&\n{w_{Na}}_{L^{1}(1-\delta_{*},1)}=\delta_{*}\bbn{\int_{y}^{1}w_{Na}'(z)dz}_{L^{\infty}(1-\delta_{*},1)}\leq \delta_{*}^{\frac{3}{2}}\n{w_{Na}'}_{L^{2}}\lesssim \delta_{*}^{\frac{3}{2}}\nu^{-1}\n{F}_{H_{k}^{-1}},\\
&\babs{\lan V''\phi_{Na}, \eta_{\delta_{*}}\ran}=\bbabs{\int^{1}_{0} \lrs{\int^{y}_{0} \phi_{Na}' dz} V'' \eta_{\delta_{*}} dy}\lesssim \delta_{*}\|\pa_y \phi_{Na}\|_{L^2}\lesssim \delta_{*} \nu^{-\frac12}|k|^{-1}\n{F}_{H_{k}^{-1}}.
\end{align*}
These above estimates, together with \eqref{equ:U-la,w,vp}, \eqref{equ:vp1,vp,decomposition}, and \eqref{equ:eta,delta,star}, imply
\begin{align}\label{inner: u-lawna,vp}
&|\lra{(V-\la)w_{Na},\vp}|\\
&\leq |\vp(1)|\lra{(V-\la)w_{Na},\eta_{\delta_{*}}}+|\lra{(V-\la)w_{Na},\vp_{1}}|\notag\\
&\lesssim |\vp(1)|\n{w_{Na}}_{L^{1}(1-\delta_{*},1)}\n{(V-\la)\eta_{\delta_{*}}}_{L^{\infty}}+
|\lra{(V-\la)w_{Na},\vp_{1}}|\notag\\
&\lesssim  |\vp(1)|
\nu^{-1}\delta_{*}^{\frac{3}{2}}\n{F}_{H_{k}^{-1}}\lrs{|V(1)-\la| + \delta_*}+|k|^{-1}\n{F}_{H_{k}^{-1}}\|\varphi-\vp(1)\eta_{\delta_{*}}(y)\|_{H_{k}^1}\notag\\
&\quad+\nu^{-\frac13}|k|^{-\frac23}\n{F}_{H_{k}^{-1}}\|\varphi-\vp(1)\eta_{\delta_{*}}(y)\|_{L^2}+|\lra{ V'' \phi_{Na},\vp-\vp(1)\eta_{\delta_{*}}}|\notag\\
&\lesssim  |\vp(1)||k|^{-1}\n{F}_{H_{k}^{-1}}\lrc{\nu^{-1}|k|\delta_{*}^{\frac{3}{2}}\lrs{|V(1)-\la| + \delta_*}+
\delta_{*}^{-\frac{1}{2}}+\nu^{-\frac13}|k|^{\frac13}\delta_{*}^{\frac{1}{2}}+\nu^{-\frac12}\delta_{*} }\notag\\
&\quad+|k|^{-1}\n{F}_{H_{k}^{-1}} \lrs{\|\pa_{y}\varphi\|_{L^{2}}+\nu^{-\frac13}|k|^{\frac13}\|\varphi\|_{L^2}}+|\lra{ V'' \phi_{Na},\vp}|,\notag
\end{align}
where in the last inequality we used $|k|\n{\vp}_{L^{2}}\lesssim \nu^{-\frac13}|k|^{\frac13}\|\varphi\|_{L^2}$.\smallskip

Now, we get into the analysis of \eqref{weak esti 1} and \eqref{weak esti 2}.\smallskip

\textbf{The proof of \eqref{weak esti 1}.}\, For $0\leq |V(1)-\la|\leq 2|k|^{-1}$, we set
\begin{align*}
\delta_{*}=\delta_{0}^{\frac32}(|V(1)-\la|+\delta_{0})^{-\frac12},
\quad \delta_{0}=\nu^{\frac{1}{3}}|k|^{-\frac{1}{3}},\quad \delta_{*}\leq \delta_{0},
\end{align*}
and hence obtain
\begin{equation}\label{equ:parameter,weak-type,case1}
\left\{
\begin{aligned}
&\nu^{-\frac13}|k|^{\frac13}\delta_{*}^{\frac{1}{2}}=\delta_{0}^{-1}\delta_{*}^{\frac12}\leq \delta_{*}^{-\frac{1}{2}},\quad
\nu^{-\frac12}\delta_{*}=|k|^{-\frac12}\delta_{0}^{-\frac32}\delta_{*}\leq  \delta_{*}^{-\frac{1}{2}},\\
&\nu^{-1}|k|\delta_{*}^{\frac{3}{2}}\lrs{|V(1)-\la| +\delta_{0} }=\delta_{*}^{-\frac{1}{2}}=\nu^{-\frac{1}{4}}|k|^{\frac{1}{4}}(|V(1)-\la|+\delta_{0})^{\frac14}.
\end{aligned}
\right.
\end{equation}
Inserting \eqref{equ:parameter,weak-type,case1} into \eqref{inner: u-lawna,vp},
we obtain
\begin{align}\label{inner: wNa,var}
|\lra{(V-\la)w_{Na},\vp}|
\lesssim & |\vp(1)|\n{F}_{H_{k}^{-1}} \nu^{-\frac14}|k|^{-\frac34}\lrs{|V(1)-\la| + \delta_{0}}^{\frac14}\\
&+|k|^{-1}\n{F}_{H_{k}^{-1}} \lrs{\|\pa_{y}\varphi\|_{L^{2}}+\delta_{0}^{-1}\|\varphi\|_{L^2}}+|\lra{ V'' \phi_{Na},\vp}|.\notag
\end{align}
Using \eqref{inner: wNa,var} with $\varphi=f\frac{\chi_{\delta_{0}}}{V-\la}$, we arrive at
\begin{equation}\label{innerproduct wf}
\begin{aligned}
|\lan w_{Na},f\ran|\leq \big|\blra{(V-\la)w_{Na},f\frac{\chi_{\delta_{0}}}{V-\la}}\big|+\lrv{\lra{w_{Na},f\chi_{\delta_{0}}^{c}}}
\leq \mathrm{I}_{1}+\mathrm{I}_{2}+\mathrm{I}_{3}+\mathrm{I}_{4},
\end{aligned}
\end{equation}
where
\begin{align*}
\mathrm{I}_{1}=& \bbabs{f(1)\frac{\chi_{\delta_{0}}(1)}{V(1)-\la}}\n{F}_{H_{k}^{-1}}\nu^{-\frac14}|k|^{-\frac34}\lrs{|V(1)-\la| + \delta_{0}}^{\frac14},\\
\mathrm{I}_{2}=&|k|^{-1}\n{F}_{H_{k}^{-1}}\lrs{\bbn{\pa_{y}\lrs{f\frac{\chi_{\delta_{0}}}{V-\la}}}_{L^{2}}+ \delta_{0}^{-1}\bbn{f\frac{\chi_{\delta_{0}}}{V-\la}}_{L^2}},\\
\mathrm{I}_{3}=&\lrv{\bblra{V'' \phi_{Na}, f\frac{\chi_{\delta_{0}}}{V-\la}}},\\
\mathrm{I}_{4}=&\lrv{\lra{w_{Na},f\chi_{\delta_{0}}^{c}}}.
\end{align*}

For $\mathrm{I}_{1}$, using the expression \eqref{fun: chi1}, we have
\begin{align*}
\bbabs{\frac{\chi_{\delta_{0}}(1)}{V(1)-\la}}\lesssim (|V(1)-\la|+\delta_{0})^{-1},
\end{align*}
which implies that
\begin{align*}
\mathrm{I}_{1}\leq \nu^{-\frac14}|k|^{-\frac34}|f(1)|\n{F}_{H_{k}^{-1}}\lrs{|V(1)-\la| + \delta_{0}}^{-\frac{3}{4}} .
\end{align*}

For $\mathrm{I}_{2}$, applying the estimate \eqref{esti: chi}, we have
\begin{align*}
\mathrm{I}_{2}
\leq&|k|^{-1}\n{F}_{H_{k}^{-1}}\lrc{\|\pa_y f\|_{L^\infty}\bbn{\frac{\chi_{\delta_{0}}}{V-\la}}_{L^2} +\|f\|_{L^\infty}\lrs{\bbn{\pa_y\lrs{\frac{\chi_{\delta_{0}}}{V-\la}}}_{L^2}
+\delta_{0}^{-1}\bbn{\frac{\chi_{\delta_{0}}}{V-\la}}_{L^2}}}\notag\\
\lesssim & |k|^{-1}\n{F}_{H_{k}^{-1}}\lrs{\nu^{-\frac16}|k|^{\frac16}\|\pa_y f\|_{L^\infty}+\nu^{-\frac12}|k|^{\frac12}\|f\|_{L^\infty}}.\notag
\end{align*}

For $\mathrm{I}_{3}$,
we first define
\begin{equation*}
\left\{
\begin{aligned}
&y_{\la}\equiv 1,\  &\la\in [V(1), V(1)+2|k|^{-1}],\\
&\text{$y_{\la}$ is the unique point such that $V(y_{\la})=\la$},\ &\la\in [V(1)-2|k|^{-1}, V(1)].
\end{aligned}
\right.
\end{equation*}
Then we have
\begin{align}\label{inner: phiNa,fchi mid}
\mathrm{I}_{3}\leq& \int^{1}_{0} \lrv{\frac{\lrc{( V'' \phi_{Na} f)(y)-( V''\phi_{Na} f)(y_{\la})} \chi_{\delta_{0}}(y)}{V(y)-V(y_{\la})}}dy +\lrv{\int^{1}_{0}\frac{( V''\phi_{Na} f)(y_{\la}) \chi_{\delta_{0}}(y)}{V(y)-V(y_{\la})}dy }\\
:=& \mathrm{I}_{3,1}+\mathrm{I}_{3,2}.\notag
  \end{align}

Noting that $|V(y)-V(y_{\la})|\sim |y-y_{\la}|$, we use H\"{o}lder's inequality and resolvent estimate \eqref{esti: wH-1} to obtain
\begin{align*}
 \mathrm{I}_{3,1}
 \leq & \int^{1}_{0} \lrv{\frac{\int_{y_{\la}}^{y}\pa_{z}(V'' \phi_{Na} f)(z)dz}{V(y)-V(y_{\la})}}dy
 \lesssim\|\pa_y ( V'' f \phi_{Na})\|_{L^2}\\
 \lesssim &\|\phi_{Na}\|_{L^2}\|\pa_y f\|_{L^\infty}+(\|\phi_{Na}\|_{L^2}+\|\pa_y \phi_{Na}\|_{L^2})\|f\|_{L^\infty}\\
 \lesssim &\nu^{-\frac12}|k|^{-1}\|F\|_{H^{-1}_{k}}(|k|^{-1}\|\pa_y f\|_{L^\infty}+\|f\|_{L^\infty}).
\end{align*}
For $\mathrm{I}_{3,2}$, if $V(1)-2|k|^{-1}\leq \la\leq V(1)-\delta_{0}$, we apply the resolvent estimate \eqref{esti: wH-1} to get
\begin{align*}
\mathrm{I}_{3,2}\lesssim &\|f\|_{L^\infty}|\phi_{Na}(y_{\la})|\lrv{\int^{1}_{0}\frac{\chi_{\delta_{0}}(y)}{V(y)-V(y_{\la})}dy}\\
\lesssim &\|f\|_{L^\infty}\lrv{\int^{1}_{y_{\la}} \pa_y \phi_{Na} dy}
\lrv{\int_{E_{\delta_{0}}}\frac{2(V(y)-\la) }{\delta_{0}^2}-\frac{(V(y)-\la)^{3}}{\delta_{0}^4} dy+
\int_{(0,1)\setminus E_{\delta_{0}}} \frac{1}{V(y)-\la} dy} \\
\lesssim & \|f\|_{L^\infty}\|\pa_y \phi_{Na}\|_{L^2} (1-y_{\la})^{\frac12} (1+|\ln(1-y_{\la})|+|\ln(y_{\la})|)\\
\lesssim & \nu^{-\frac12}|k|^{-1}\|F\|_{H^{-1}_{k}}\|f\|_{L^\infty}.
\end{align*}

If $V(1)-\delta_{0}\leq \la \leq V(1)$, let $y_{\la,\delta_{0}}$ be the unique point such that $V(y_{\la,\delta_{0}})=\la-\delta_{0}$.
By the resolvent estimate \eqref{esti: wH-1}, we have
\begin{align*}
\mathrm{I}_{3,2}\lesssim &\|f\|_{L^\infty}|\phi_{Na}(y_{\la})|\lrv{\int^{1}_{0}\frac{\chi_{\delta_{0}}(y)}{V(y)-V(y_{\la})}dy}\\
\lesssim &\|f\|_{L^\infty}\lrv{\int^{1}_{y_{\la}} \pa_y \phi_{Na} dy} \lrv{\int_{y_{\la,\delta_{0}}}^{1}\frac{2(V(y)-\la) }{\delta_{0}^2}-\frac{(V(y)-\la)^{3}}{\delta_{0}^4} dy+
\int_{0}^{y_{\la,\delta_{0}}} \frac{1}{V(y)-\la} dy} \\
\lesssim & \|f\|_{L^\infty}\|\pa_y \phi_{Na}\|_{L^2} (1-y_{\la})^{\frac12} (1+|\ln(y_{\la,\delta_{0}}-y_{\la})|+|\ln(y_{\la})|)\\
\lesssim & \nu^{-\frac12}|k|^{-1}\|F\|_{H^{-1}_{k}}\|f\|_{L^\infty}.
\end{align*}
where in the last inequality we also used that $y_{\la}-y_{\la,\delta_{0}}\sim \delta_{0}$ and $1-y_{\la}\lesssim \delta_{0}$.

Combining the above estimates for $\mathrm{I}_{3,1}$ and $\mathrm{I}_{3,2}$, we deduce
\begin{align*}
\mathrm{I}_{3}\lesssim \nu^{-\frac12}|k|^{-1}\|F\|_{H^{-1}_{k}}(|k|^{-1}\|\pa_y f\|_{L^\infty}+\|f\|_{L^\infty}).
\end{align*}

For $\mathrm{I}_{4}$, by \eqref{esti: wH-1} and $\|\chi_{\delta_{0}}^{c}\|_{L^2}\leq \delta_{0}^{\frac12}=\nu^{\frac16}|k|^{-\frac16}$, we have
\begin{align*}
\mathrm{I}_{4}=\lrv{\lra{w_{Na},f\chi_{\delta_{0}}^{c}}}\leq \|w_{Na}\|_{L^2}\|\chi_{\delta_{0}}^{c}\|_{L^2}\|f\|_{L^\infty(E_{\delta_{0}})}\lesssim  \nu^{-\frac12}|k|^{-\frac12}\|F\|_{H^{-1}_{k}}\|f\|_{L^\infty(E_{\delta_{0}})}.
\end{align*}

Inserting estimates for $\mathrm{I}_{1}$--$\mathrm{I}_{4}$ into \eqref{innerproduct wf}, we obtain
\begin{align*}
|\lan w_{Na},f\ran|\lesssim &\nu^{-\frac14}|k|^{-\frac34}\lrv{f(1)}\n{F}_{H_{k}^{-1}}\lrs{|V(1)-\la| + \delta_{0}}^{-\frac34} \\
&+\nu^{-\frac12}|k|^{-\frac12}\n{F}_{H_{k}^{-1}} \lrs{|k|^{-1}\|\pa_y f\|_{L^\infty}+\|f\|_{L^\infty}},
\end{align*}
which completes the proof of \eqref{weak esti 1}.\smallskip

\textbf{The proof of \eqref{weak esti 2}.}\,
For $ |V(1)-\la|\geq 2|k|^{-1}$, by choosing
\begin{align*}
 \delta_{*}=\nu^{\frac12}|k|^{-\frac12}(|V(1)-\la|+|k|^{-1})^{-\frac12},\quad \delta_{1}=|k|^{-1},\quad \delta_{*}\leq \delta_{1},
 \end{align*}
we have
\begin{equation}\label{sim esti}
\left\{
\begin{aligned}
&\nu^{-1}|k|\delta_{*}^{\frac{3}{2}}\lrs{|V(1)-\la| + |k|^{-1}}=\delta_{*}^{-\frac{1}{2}}
=\nu^{-\frac14}\lrs{|k(V(1)-\la)| + 1}^{\frac14},\\
&\nu^{-\frac13}|k|^{\frac13}\delta_{*}^{\frac{1}{2}}\leq\nu^{-\frac14}\lrs{|k(V(1)-\la)| + 1}^{-\frac14},\quad
\nu^{-\frac{1}{2}}\delta_{*}=\lrs{|k(V(1)-\la)| + 1}^{-\frac12}.
\end{aligned}
\right.
\end{equation}

Inserting \eqref{sim esti}  into \eqref{inner: u-lawna,vp} with $\varphi=f\frac{\chi_{\delta_{1}}}{V(y)-\la}$, we obtain
\begin{align}\label{inner wna,fJ}
|\lan w_{Na},f\ran|
\leq \lrv{\blra{(V-\la)w_{Na},\frac{f\chi_{\delta_{1}}}{V-\la}}}+\lrv{\lra{w_{Na},f\chi_{\delta_{1}}^{c}}}\lesssim \mathrm{J}_{1}+\mathrm{J}_{2}+\mathrm{J}_{3}+\mathrm{J}_{4},
\end{align}
where
\begin{align*}
\mathrm{J}_{1}=&|k|^{-1}\n{F}_{H_{k}^{-1}}\bbabs{f(1)\frac{\chi_{\delta_{1}}(1)}{V(1)-\la}}\nu^{-\frac14}\lrs{|k(V(1)-\la)| + 1}^{\frac14} ,\\
\mathrm{J}_{2}=&|k|^{-1}\n{F}_{H_{k}^{-1}}\lrs{\bbn{\pa_{y}\lrs{f\frac{\chi_{\delta_{1}}}{V-\la}}}_{L^{2}}+ \nu^{-\frac13}|k|^{\frac13}\bbn{f\frac{\chi_{\delta_{1}}}{V-\la}}_{L^2}},\\
\mathrm{J}_{3}=&\lrv{\bblra{V'' \phi_{Na}, f\frac{\chi_{\delta_{1}}}{V-\la}}},\\
\mathrm{J}_{4}=&\lrv{\lra{w_{Na},f\chi_{\delta_{1}}^{c}}}.
\end{align*}

For $\mathrm{J}_{1}$, using the condition that $|V(1)-\la|\geq 2|k|^{-1}$, we have
\begin{align*}
\bbabs{\frac{\chi_{\delta_{1}}(1)}{V(1)-\la}}= \frac{1}{|V(1)-\la|}\lesssim |k||(|k(V(1)-\la)|+1)^{-1},
\end{align*}
which yields that
\begin{align}
\mathrm{J}_{1}\lesssim \nu^{-\frac14}\n{F}_{H_{k}^{-1}}|f(1)|(|k(V(1)-\la)|+1)^{-\frac34}.
\end{align}

The definition for $\mathrm{J}_{2}$ already matches the desired estimate. We therefore proceed to analyse $\mathrm{J}_{3}$. By \eqref{esti: wH-1}, we have
\begin{align}\label{esti: phinavp}
\mathrm{J}_{3}\leq \|\phi_{Na}\|_{L^2}\bbn{f\frac{\chi_{\delta_{1}}}{V-\la}}_{L^2}\lesssim \nu^{-\frac12}|k|^{-2}\|F\|_{H^{-1}_{k}}\bbn{f\frac{\chi_{\delta_{1}}}{V-\la}}_{L^2}.
\end{align}

For $\mathrm{J}_{4}$, the condition $|V(1)-\la|\geq 2|k|^{-1}$ implies $\chi^{c}_{\delta_{1}}(1)=0$.
Then, together with the boundary condition $f(0)=0$ and the resolvent estimate \eqref{esti: wH-1}, we derive
\begin{align*}
\mathrm{J}_{4}\leq &\|(\pa_y, k)\phi_{Na}\|_{L^2}\|(\pa_y,k) [f\chi^{c}_{\delta_{1}}]\|_{L^2}\\
\lesssim& \nu^{-\frac12}|k|^{-1}\|F\|_{H_{k}^{-1}}\|(\pa_y,k) [f\chi^{c}_{\delta_{1}}]\|_{L^2}\\
\lesssim& \nu^{-\frac12}|k|^{-1}\|F\|_{H_{k}^{-1}}\n{(\pa_{y},k)f}_{L^{2}(E_{\delta_{1}})},
\end{align*}
where in the last inequality we used $\n{\chi^{c}_{\delta_{1}}}_{L^{\infty}}\lesssim 1$ and $\n{\pa_{y}\chi^{c}_{\delta_{1}}}_{L^{\infty}}\lesssim \delta_{1}^{-1}=|k|$.

Putting the estimates for $\mathrm{J}_{1}$--$\mathrm{J}_{4}$ into \eqref{inner wna,fJ}, we deduce
\begin{align}
|\lan w_{Na},f\ran|
\lesssim& |k|^{-1}\|F\|_{H_{k}^{-1}}\Big[ \nu^{-\frac14}|k||f(1)|(|k(V(1)-\la)|+1)^{-\frac34}+\Big\|\pa_{y}\lrs{f\frac{\chi_{\delta_{1}}}{V-\la}}\Big\|_{L^{2}}\\
& +(\nu^{-\frac13}|k|^{\frac13}+\nu^{-\frac12} k^{-1})\Big\|f\frac{\chi_{\delta_{1}}}{V-\la}\Big\|_{L^2} +\nu^{-\frac12}\n{(\pa_{y},k)f}_{L^{2}(E_{\delta_{1}})}\Big].\notag
\end{align}
Hence, we complete the proof of \eqref{weak esti 2}.
\end{proof}

Building upon the resolvent estimates in Lemma \ref{pro: reso esti FL2} and the weak-type resolvent estimates in Lemma \ref{lemma: weak resolvent estimate},
we now derive the refined bounds for the coefficients $c_1$ and $c_2$.

\begin{lemma}\label{lemma: c1,c2 FL2}
Under the same conditions as in Lemma \ref{lemma: weak resolvent estimate}, we have
\begin{align}
&(1+|k(\la-V(0))|)^{\frac34}|c_{1}|+(1+|k(\la-V(1))|)^{\frac34}|c_{2}|\lesssim \nu^{-\frac{1}{2}}|k|^{-\frac12}\|F\|_{H_{k}^{-1}},\label{c1H-1}\\
&(1+|k( \la-V(0))|)^{\frac34}|c_1|+(1+|k( \la-V(1))|)^{\frac34}|c_2|\lesssim \nu^{-\frac16}|k|^{-\frac56}\|F\|_{L^2}, \label{c1FL2}\\
&(1+|k( \la-V(0))|)|c_1|+(1+|k( \la-V(1))|)|c_2|\lesssim \nu^{-\frac{1}{12}}|k|^{-\frac53}\|F\|_{H_{k}^{1}}.\label{c1 FH1}
\end{align}
\end{lemma}
\begin{proof}
We restrict attention to the coefficient $c_{1}$, since $c_{2}$ can be handled similarly.
We divide the proof according to the regularity of $F$: $F\in H_{k}^{-1}$, $L^2$ or $ H^1_k$.
For simplicity, we set
\begin{align*}
f(y)=\frac{\sinh (k(1-y))}{\sinh k}.
\end{align*}

\noindent\textbf{Case 1. $F\in H_{k}^{-1}$}.\smallskip

{\it Case 1.1. $\la\leq V(0)-2|k|^{-1}$}.

In this case, we have $E_{\delta_{1}}=\emptyset$ and $\chi_{\delta_{1}}=1$. Using \eqref{weak esti 2} and $\nu k^{2}\lesssim 1$, we have
\begin{align}\label{esti: c1 small}
&|c_1|=|\lan w_{Na}, f\ran|\\
\lesssim &|k|^{-1}\|F\|_{H^{-1}_{k}}\Big(\nu^{-\frac14}|k|(|k(V(0)-\la)|+1)^{-\frac34}
+\Big\|\pa_{y}\lrs{\frac{f}{V-\la}}\Big\|_{L^{2}}+\nu^{-\frac{1}{2}}\Big\|\frac{f}{V-\la}\Big\|_{L^2}\Big).\notag
\end{align}
By Lemma \ref{lemma: sinh}, with $\nu k^{2}\lesssim 1$, we get
\begin{align}\label{esti: fchiH1}
\Big\|\pa_{y}\lrs{\frac{f}{V-\la}}\Big\|_{L^{2}}\leq &\|\pa_y f\|_{L^2}\Big\|\frac{1}{V-\la}\Big\|_{L^\infty} +\|f\|_{L^2}\Big\|\pa_y\Big(\frac{1}{V-\la}\Big)\Big\|_{L^\infty}\\
\lesssim &|k|^{\frac12}|V(0)-\la|^{-1}+|k|^{-\frac12}|V(0)-\la|^{-2}\notag\\
\lesssim &|k|^{\frac32}(1+|k(V(0)-\la)|)^{-1}\lesssim \nu^{-\frac12}|k|^{\frac12}(1+|k(V(0)-\la)|)^{-1},\notag
\end{align}
and
\begin{align}\label{esti: fchiL2}
\nu^{-\frac12}\Big\|\frac{f}{V-\la}\Big\|_{L^2}\leq \nu^{-\frac12}\|f\|_{L^2}\Big\|\frac{1}{V-\la}\Big\|_{L^\infty}
\leq & \nu^{-\frac12}|k|^{-\frac12}|V(0)-\la|^{-1}\\
\leq &\nu^{-\frac12}|k|^{\frac12}(1+|k(V(0)-\la)|)^{-1}.\notag
\end{align}
Inserting \eqref{esti: fchiH1} and \eqref{esti: fchiL2} into \eqref{esti: c1 small}, we obtain
\begin{align}\label{case1c1}
|c_1|\lesssim \nu^{-\frac12}|k|^{-\frac12}(1+|k(V(0)-\la)|)^{-\frac34}\|F\|_{H^{-1}_{k}}.
\end{align}

{\it Case 1.2. $|V(0)-\la|\in[0, 2|k|^{-1}]$}.

Noting that $1\sim (1+|k(V(0)-\la)|)$ in this case, the weak-type resolvent estimate \eqref{weak esti 1} yields
\begin{align}\label{case3.2c1}
|c_1|=|\lan w_{Na}, f\ran|\lesssim& \nu^{-\frac14}|k|^{-\frac34}\lrs{|V(0)-\la| + \delta_{0}}^{-\frac34} \n{F}_{H_{k}^{-1}}\\
&+\nu^{-\frac12}|k|^{-\frac12}\n{F}_{H_{k}^{-1}} \lrs{|k|^{-1}\|\pa_y f\|_{L^\infty}+\|f\|_{L^\infty}}\notag\\
\lesssim&\nu^{-\frac12}|k|^{-\frac12}(1+|k(V(0)-\la)|)^{-\frac34}\|F\|_{H^{-1}_{k}},\notag
\end{align}
where in the last inequality we used that
\begin{align*}
&\nu^{-\frac14}|k|^{-\frac34}\lrs{|V(0)-\la| + \delta_{0}}^{-\frac34}\leq \nu^{-\frac14}|k|^{-\frac34}\delta_{0}^{-\frac{3}{4}}=
\nu^{-\frac12}|k|^{-\frac12},\\
&\|\pa_y f\|_{L^\infty}\lesssim |k|, \quad \|f\|_{L^\infty}\leq 1,\quad 1\sim (1+|k(V(0)-\la)|)^{-\frac{3}{4}}.
\end{align*}

{\it Case 1.3. $\la\geq V(0)+2|k|^{-1}$}.

Using the weak-type resolvent estimate \eqref{weak esti 2} and $\nu k^{2}\lesssim 1$, we have
\begin{align}\label{case3.3innerwnaf}
|c_1|=|\lan w_{Na},f\ran|
\lesssim& |k|^{-1}\|F\|_{H_{k}^{-1}}\Big[ \nu^{-\frac14}|k|(|k(V(0)-\la)|+1)^{-\frac34}+\Big\|\pa_{y}\lrs{f\frac{\chi_{\delta_{1}}}{V-\la}} \Big\|_{L^{2}}\\
&+\nu^{-\frac12}\Big\|f\frac{\chi_{\delta_{1}}}{V-\la}\Big\|_{L^2} +\nu^{-\frac12}\n{(\pa_{y},k)f}_{L^{2}(E_{\delta_{1}})}\Big].\notag
\end{align}

For a more refined analysis, we introduce a separation point defined by
\begin{equation*}
\left\{
\begin{aligned}
&y_{*}\equiv 1,\  &\la\in [2V(1)-V(0),\infty),\qquad\qquad \\
&\text{$y_{*}$ is the unique point: $\la-V(0)=2(\la-V(y_{*}))$},\ &\la\in[V(0)+2|k|^{-1},2V(1)-V(0)].
\end{aligned}
\right.
\end{equation*}
From this definition, it follows that
\begin{align}\label{equ:U-la,Linf}
\bbn{\frac{1}{V-\la}}_{L^{\infty}(0,y_{*})}\lesssim |\la-V(0)|^{-1}.
\end{align}
For $y_{*}<1$, due to that $2(V(y_{*})-V(0))=\la-V(0)$, we have $y_{*}\sim \la-V(0)$, and hence derive
\begin{align}
\n{f}_{L^{\infty}(y^{*},1)}=& \bbn{\frac{\sinh (k(1-y))}{\sinh k}}_{L^{\infty}(y^{*},1)}\lesssim e^{-|ky_{*}|}\lesssim |ky_{*}|^{-1}\lesssim |k(\la-V(0))|^{-1},\label{equ:f,Linf}\\
\n{\pa_{y}f}_{L^{\infty}(y^{*},1)}\lesssim&|k|\bbn{\frac{\cosh (k(1-y))}{\sinh k}}_{L^{\infty}(y^{*},1)}\lesssim |k|e^{-|ky_{*}|}\lesssim |y_{*}|^{-1}\lesssim |\la-V(0)|^{-1}.\label{equ:pa,f,Linf}
\end{align}

By \eqref{esti: chi}, \eqref{equ:U-la,Linf}, \eqref{equ:f,Linf}, and Lemma \ref{lemma: sinh}, we get
\begin{align*}
\Big\|f\frac{\chi_{\delta_{1}}}{V-\la}\Big\|_{L^2}
\leq &\|f\|_{L^2}\bbn{\frac{1}{V-\la}}_{L^{\infty}(0,y_{*})}+\|f\|_{L^\infty(y_{*},1)}\Big\| \frac{\chi_{\delta_{1}}}{V-\la}\Big\|_{L^2}\\
\lesssim &|k|^{-\frac{1}{2}}|\la-V(0)|^{-1}+\delta_{1}^{-\frac{1}{2}}|k(\la-V(0))|^{-1}\lesssim|k|^{\frac12}|k(\la-V(0))|^{-1}.
\end{align*}
Similarly, using \eqref{esti: chi}, \eqref{equ:U-la,Linf}--\eqref{equ:pa,f,Linf} and Lemma \ref{lemma: sinh}, we also have
\begin{align*}
\bbn{f\pa_y\lrs{\frac{\chi_{\delta_{1}}}{V-\la}}}_{L^2}+\bbn{\frac{\chi_{\delta_{1}}}{V-\la}\pa_y f}_{L^2}\lesssim |k|^{\frac32}|k(\la-V(0))|^{-1},
\end{align*}
which, together with $\nu k^{2}\lesssim 1$, implies that
\begin{align}
\Big\|\pa_{y}\lrs{f\frac{\chi_{\delta_{1}}}{V-\la}}\Big\|_{L^{2}}+\nu^{-\frac12}\Big\|f\frac{\chi_{\delta_{1}}}{V-\la}\Big\|_{L^2}\lesssim & \nu^{-\frac12}|k|^{\frac12}|k(\la-V(0))|^{-1}. \label{case3 fchiH1}
\end{align}

Similar to \eqref{equ:f,Linf} and \eqref{equ:pa,f,Linf}, we obtain
 \begin{align*}
 \|(\pa_y,k) f\|_{L^2(y^{*},1)}\lesssim |k|\|e^{-|ky|}\|_{L^2(y^{*},1)}\lesssim |k|^{\frac12} e^{-|ky_{*}|}\lesssim |k|^{\frac12}|k(\la-V(0))|^{-1},
 \end{align*}
 which, together with $E_{\delta_{1}}\subset (y_{*},1)$, implies
 \begin{equation}
\begin{aligned}\label{case3finfty}
\nu^{-\frac12}\n{(\pa_{y},k)f}_{L^{2}(E_{\delta_{1}})}
\lesssim \nu^{-\frac12}|k|^{\frac12}|k(\la-V(0))|^{-1}.
\end{aligned}
\end{equation}

Inserting  \eqref{case3 fchiH1}  and \eqref{case3finfty} into \eqref{case3.3innerwnaf}, with $(\la-V(0))\geq 2|k|^{-1}$
we derive
\begin{align}\label{case3c1}
(1+|k(\la-V(0))|)^{\frac34}|c_1|\lesssim \nu^{-\frac12}|k|^{-\frac12}\|F\|_{H^{-1}_{k}}.
\end{align}

\noindent\textbf{Case 2. $F\in L^2$}.

The analysis in this case follows a similar approach to that for $F \in H^{-1}_{k}$. Using the same argument, we also obtain
\begin{align*}
&(1+|k( \la-V(0))|)^{\frac34}|c_1|\lesssim \nu^{-\frac16}|k|^{-\frac56}\|F\|_{L^2}.
\end{align*}

\noindent\textbf{Case 3. $F\in H^{1}_{k}$.}\smallskip

{\it Case 3.1. $|\la-V(0)|\leq |k|^{-1}$.}

It directly follows from \eqref{c1c2H1} that
\begin{align*}
(1+|k( \la-V(0))|)|c_1| \lesssim \nu^{-\frac{1}{12}}|k|^{-\frac53}\|F\|_{H^{1}_{k}}.
\end{align*}

{\it Case 3.2. $V(0)-\la\geq |k|^{-1}$.}

By Lemma \ref{lemma: sinh} and \eqref{esti: wH1}, we have
\begin{align*}
|c_1|=&\lra{w_{Na},f}\leq \n{f}_{L^{2}}\bbn{\frac{1}{V-\la}}_{L^{\infty}}\|(V-\la)w_{Na}\|_{L^2}\\
\lesssim &|k|^{-\frac52}(\la-V(0))^{-1}\|F\|_{H^{1}_{k}}\lesssim |k|^{-\frac32}(1+|k(\la-V(0))|)^{-1}\|F\|_{H^{1}_{k}}.
\end{align*}

\textit{Case 3.3. $\la-V(0)\geq |k|^{-1}$.}

Using the bounds \eqref{equ:U-la,Linf}, \eqref{equ:f,Linf}, together with  \eqref{lemma: sinh} and \eqref{esti: wH1}, we get
\begin{align*}
|c_1|=\lra{w_{Na},f}
\leq& \n{f}_{L^{2}}\bbn{\frac{1}{V-\la}}_{L^{\infty}(0,y_{*})}\|(V-\la)w_{Na}\|_{L^2} +\| f\|_{L^\infty(y_{*},1)}\|w_{Na}\|_{L^1}\\
\lesssim & \nu^{-\frac{1}{12}}|k|^{-\frac53}|k(\la-V(0))|^{-1}\|F\|_{H^{1}_{k}}\\
\lesssim & \nu^{-\frac{1}{12}}|k|^{-\frac53}(1+|k (\la-V(0))|)^{-1}\|F\|_{H^{1}_{k}},
\end{align*}
where in the last inequality we used the condition that $\la-V(0)\geq |k|^{-1}$.
\end{proof}

We are now in a position to conclude the resolvent estimates for the vorticity under no-slip boundary condition.

\begin{proof}[\textbf{Proof of Proposition \ref{lemma:non-slip boundary,resolvent}}]
We notice that
\begin{align*}
&\|w\|_{L^2}\leq  \|w_{Na}\|_{L^2}+|c_1|\|w_1\|_{L^2}+|c_2|\|w_2\|_{L^2},\\
&\|\rho^{\frac12}_k w\|_{L^2}\leq  \|w_{Na}\|_{L^2}+|c_1|\|\rho^{\frac12}_kw_1\|_{L^2}+|c_2|\|\rho^{\frac12}_kw_2\|_{L^2}.
\end{align*}
The estimates \eqref{esti: u,wL2FL^2}--\eqref{esti: u,wL2FH-1} can be directly proven by resolvent estimates \eqref{esti: wL2}--\eqref{esti: wH-1} with the Navier-slip boundary condition, boundary layer correctors estimates \eqref{w1L2}--\eqref{weightL2}, and the coefficient estimates \eqref{c1H-1}--\eqref{c1 FH1}.
\end{proof}

\section{Space-Time estimates of the linearized NS}\label{sec: space-time estimate}

In this section, we establish the space-time estimates for the linearized Navier-Stokes system
with no-slip boundary condition
\begin{equation}\label{equ: om}
\left\{\begin{aligned}
&\pa_t\om-\nu(\pa^2_y-k^2)\om+ik U\om-ik \pa^2_y U\psi=-i k f_1-\pa_y f_2,\\
& (\pa^2_y-\al^2)\psi=\om,\quad \psi(t,0)=\psi(t,1)=\psi'( t, 0)=\psi'( t, 1)=0,\quad \om(0,y)=\om_{k}^{\mathrm{in}}(y),
\end{aligned}
\right.
\end{equation}
where we omit the subscript $k$ for brevity.
Also, we set the notation
 $\n{f}_{L^{p}L^{q}}=\n{f}_{L_{t}^{p}L_{y}^{q}}$.

\begin{proposition}\label{lemma:est,u,om,timespace,f12}
Let $(\om,\psi)$ be the solution to \eqref{equ: om}. Then there exist constants $\nu_0>0$ and $\epsilon_0>0$ such that, for $\nu\in(0,\nu_0]$, $\epsilon\in(0,\epsilon_{0}]$, it holds that
\begin{align*}
&|k|^2\| e^{\epsilon \nu^{\frac13}t} u\|^2_{L^2L^2}+\nu^{\frac12}|k|\|e^{\epsilon \nu^{\frac13}t} \om\|^2_{L^2L^2}
+\|e^{\epsilon \nu^{\frac13}t} \sqrt{1-(2y-1)^2}\om\|^2_{L^\infty L^2}+\|e^{\epsilon \nu^{\frac13}t} u\|^2_{L^\infty L^\infty}\\
&\quad\lesssim  E^{\mathrm{in}}+\nu^{-1}\|e^{\epsilon \nu^{\frac13}t}(f_1,f_2)\|^2_{L^2L^2},
\end{align*}
where $E^{\mathrm{in}}=:|k|^{-2} \|\pa_y\om^{\mathrm{in}}\|^2_{L^2}+\|u^{\mathrm{in}}\|^2_{H^1}$.
\end{proposition}

Here, we prove the inviscid damping and enhanced dissipation estimates by using the refined resolvent estimates obtained in Section \ref{sec: Non slip}.
Building on these estimates, we further derive the weighted space-time $L^{\infty} L^{2}$ estimates for the vorticity and $L^{\infty}L^{\infty}$ estimates for the velocity.
The proof is divided into the high frequency and low frequency cases, that is, $\nu k^2\geq \n{\pa_{y}U}_{L_{t}^{\infty}L_{y}^{\infty}}$ and $\nu k^2\leq \n{\pa_{y}U}_{L_{t}^{\infty}L_{y}^{\infty}}$. The detailed analysis for the high frequency is carried out in Subsections \ref{subsec: hingfre}.
 For low frequency case, we first establish space-time estimates for the steady shear flow in Subsection \ref{subsec: lowfre}, and then extend the analysis to the time-dependent heat flow case via the frozen-time method in Subsection \ref{sec:Space-time estimates for low frequencies: heat flow case}.
  Finally in Subsection \ref{sec: proof full pro}, we complete the proof of Proposition \ref{equ: om}.

\subsection{Space-time estimates for high frequencies}\label{subsec: hingfre}
\begin{proposition}\label{Th xi small}
Let $\nu k^2\geq \n{\pa_{y}U}_{L_{t}^{\infty}L_{y}^{\infty}}$ and $(\om,\psi)$ be the solution to \eqref{equ: om}. Then
there exist constants $\nu_0>0$ and $\epsilon_0>0$ such that, for $\nu\in(0,\nu_0]$, $\epsilon\in[0,\epsilon_{0}]$, it holds that
\begin{align}\label{equ:high frequency,space-time,flow}
&|k|^2\|e^{\epsilon \nu^{\frac13}t}u\|^2_{L^2 L^2}+\nu |k|^2\|e^{\epsilon \nu^{\frac13}t}\om\|^2_{L^2 L^2}+\|e^{\epsilon \nu^{\frac13}t}\om\|^2_{L^\infty L^2}+|k|\|e^{\epsilon \nu^{\frac13}t}u\|^2_{L^\infty L^\infty}\\
\quad&\lesssim E^{\mathrm{in}}+\nu^{-1}\|e^{\epsilon \nu^{\frac13}t}(f_1,f_2)\|^2_{L^2 L^2}.\notag
\end{align}
\end{proposition}

\begin{proof}
We divide the proof into two steps.\smallskip

\textbf{Step 1.} The estimates for $|k|^2\|e^{\epsilon \nu^{\frac13}t}u\|^2_{L^2 L^2}+\nu |k|^2\|e^{\epsilon \nu^{\frac13}t}\om\|^2_{L^2 L^2}$.\smallskip

Testing \eqref{equ: om} by $\psi$ and taking the real part, with $\nu k^2\geq \n{\pa_{y}U}_{L_{t}^{\infty}L_{y}^{\infty}}$, we get
\begin{align*}
\frac12\frac{d}{dt}\|u\|^2_{L^2}+\nu\|\om\|^2_{L^2}
=&-\operatorname{Re}\left(i k \int^1_{0}\pa_y U(t, y)\psi'\overline{\psi} dy\right)+\operatorname{Re}(\lan i k f_1,\psi\ran-\lan f_2,\pa_y \psi\ran)\\
\leq &\lrs{\frac12+\eta} \nu k^{2}\|(\pa_y,k)\psi\|^2_{L^2}+C_{\eta}(\nu k^{2})^{-1}\lrs{\|f_1\|^2_{L^2}+\|f_2\|^2_{L^2}},
\end{align*}
which yields that
\begin{align*}
&\frac{d}{dt}\|e^{\epsilon \nu^{\frac13}t} u\|^2_{L^2}+\nu\|e^{\epsilon \nu^{\frac13}t}\om\|^2_{L^2}\\
&\leq C_{\eta}(\nu k^2)^{-1}\|e^{\epsilon \nu^{\frac13}t}(f_1,f_2)\|^2_{L^2}+\lrs{\epsilon |k|^{-\frac23}+\frac12+\eta} \nu k^{2}
\|e^{\epsilon \nu^{\frac12}|k|} u\|^2_{L^2}.
\end{align*}

Noticing that
\begin{align}\label{uL2leqomL2}
\|(\pa_y,k)\psi\|^2_{L^2}=\|u\|^2_{L^2}\leq |k|^{-2}\|\om\|^2_{L^2}, 
\end{align}
by choosing $\eta\ll 1$
and then integrating with respect to $t$, we obtain
\begin{align}\label{nuxi2geq1 uL2omL2}
|k|^2\|e^{\epsilon \nu^{\frac13}t} u\|^2_{L^\infty L^2}+\nu |k|^{2} \|e^{\epsilon \nu^{\frac13}t}\om\|^2_{L^2L^2}\lesssim |k|^2\| u^{\mathrm{in}}\|^2_{L^2}+\nu^{-1}\|e^{\epsilon \nu^{\frac13}t}(f_1, f_2)\|^2_{L^2L^2}.
\end{align}

By \eqref{uL2leqomL2} and \eqref{nuxi2geq1 uL2omL2}, we also have
\begin{align}\label{nuxi2geq1 uLtinftyuLt2}
|k|^2\|e^{\epsilon \nu^{\frac13}t} u\|^2_{L^2L^2}\lesssim |k|^2\| u^{\mathrm{in}}\|^2_{L^2}+\nu^{-1}\|e^{\epsilon \nu^{\frac13}t}(f_1, f_2)\|^2_{L^2L^2}.
\end{align}

\textbf{Step 2.} The estimates for $\|e^{\epsilon \nu^{\frac13}t}\om\|^2_{L^\infty L^2}+|k|\|e^{\epsilon \nu^{\frac13}t}u\|^2_{L^\infty L^\infty}$.\smallskip

Let $F_1=\pa_t \psi+ik U\psi$. It satisfies
\begin{align*}
(\pa^2_y-k^2)F_1=&\pa_t(\pa^2_y-k^2)\psi+ik U(\pa^2_y-k^2)\psi+2ik \pa_y U\pa_y \psi+ik \pa^2_y U\psi\\
=&\pa_t \om+ik U\om+2ik \pa_y U\pa_y\psi+ik\pa^2_y U\psi,
\end{align*}
where $F_1|_{y=0,1}=\pa_y F_1|_{y=0,1}=0$. Integrating by parts and using the equation \eqref{equ: om}, we have
\begin{align*}
&\lan i k f_1, F_1\ran-\lan f_2,\pa_y F_1\ran=\lan -i k f_1-\pa_y f_2, -F_1\ran \\
&=\lan (\pa_t-\nu(\pa^2_y- k^2)+i k U)\om-ik \pa^2_y U \psi,-F_1\ran\\
&=\lan (\pa^2_y- k^2)F_1-2i k\pa_y U \pa_y\psi-\nu(\pa^2_y- k^2)\om-2ik \pa^2_y U\psi,-F_1\ran\\
&=\|(\pa_y,k) F_1\|^2_{L^2}+\lan 2i k \pa_y U \pa_y\psi, F_1\ran\\
&\quad+\nu \lan\om,\pa_t\om+i k U \om+2i k\pa_y U \pa_y\psi +ik\pa^2_y U\psi\ran+2ik\lan \pa^2_y U\psi,F_1\ran.
\end{align*}
Taking the real part and then using Cauchy-Schwarz inequality, we obtain
\begin{align}\label{dtomL2+FL2}
&\frac{\nu}{2}\frac{d}{dt}\|\om\|^2_{L^2}+\|(\pa_y,k) F_1\|^2_{L^2}\\
&\leq 2\nu | k||\lan \om, \pa_y U \pa_y\psi\ran|+\nu|k|\lan \pa_y \psi, \pa^3_y U\psi\ran+2| k||\lan \pa_yU \pa_y\psi+\pa^2_yU\psi+f_1, F_1\ran|+|\lan f_2,\pa_y F_1\ran|\notag\\
&\leq\nu \|\om\|^2_{L^2}+\nu \|\pa^3_y U\|_{L^\infty}\|u\|^2_{L^2}+(\nu k^2+4)\|\pa_y U\|^2_{L^\infty}\|\pa_y\psi\|^2_{L^2}+4\|\pa^2_yU\|^2_{L^\infty}\|\psi\|^2_{L^2}\notag\\
&\quad+\frac34|k|^2\|F_1\|^2_{L^2}+\frac14\|\pa_y F_1\|^2_{L^2}+4\|f_1\|^2_{L^2}+\|f_2\|^2_{L^2}.\notag
\end{align}
By \eqref{uL2leqomL2}, we have
\begin{align*}
\nu \frac{d}{dt}\|\om\|^2_{L^2}\lesssim \|(f_1,f_2)\|^2_{L^2}+\nu \|\om\|^2_{L^2},
\end{align*}
which implies that
\begin{align*}
\frac{d}{dt}\|e^{\epsilon \nu^{\frac13}t} \om\|^2_{L^2}\lesssim \nu^{-1}\|e^{\epsilon \nu^{\frac13}t} (f_1,f_2)\|^2_{L^2}+(1+\epsilon\nu^{\frac13}) \|e^{\epsilon \nu^{\frac13}t} \om\|^2_{L^2}.
\end{align*}
Thanks to \eqref{nuxi2geq1 uL2omL2}, with $\nu k^{2}\gtrsim 1$, we get
\begin{align*}
\|e^{\epsilon \nu^{\frac13}t} \om\|^2_{L^\infty L^2}
\lesssim & \|\om^{\mathrm{in}}_{k}\|^2_{L^2}+\nu^{-1}\|e^{\epsilon \nu^{\frac13}t} (f_1, f_2)\|^2_{L^2L^2}+\nu|k|^{2} \|e^{\epsilon \nu^{\frac13}t} \om\|^2_{L^2L^2}\\
\lesssim & \|\om^{\mathrm{in}}_{k}\|^2_{L^2}+\nu^{-1}\|e^{\epsilon \nu^{\frac13}t} (f_1, f_2)\|^2_{L^2L^2}.
\end{align*}

Due to $\|u\|_{L^\infty}\leq |k|^{-\frac12}\|\om\|_{L^2}$, we also have
\begin{align*}
|k|\|e^{\epsilon \nu^{\frac13}t}u\|^2_{L^\infty L^\infty}\lesssim& \|\om^{\mathrm{in}}_{k}\|^2_{L^2}+ \nu^{-1}\|e^{\epsilon \nu^{\frac13}t}(f_1,f_2)\|^2_{L^2L^2}.
\end{align*}

Therefore, we have completed the proof of Proposition \ref{Th xi small}.
\end{proof}
\subsection{Space-time estimates for low frequencies: frozen-time case}\label{subsec: lowfre}
We first consider the frozen-coefficient equation
\begin{equation}\label{equ: om const}
\left\{
\begin{aligned}
&\pa_t \om -\nu(\pa^2_y- k^2)\om +i k V\om -ik V''\psi =-i k f_1-\pa_y f_2-f_3-f_4, \\
&\om=(\pa^2_y- k^2)\psi,\quad \psi(0)=\psi(1)=\psi'(0)=\psi'(1)=0,\\
&\om(0,y)=\om_{k}^{\mathrm{in}}(y),
\end{aligned}
\right.
\end{equation}
where the steady flow satisfies
\begin{align*}
V(y)\in H^{4}(0,1), \quad \inf_{y\in(0,1)}\pa_{y}V\geq c_{0}>0,\,\, \text{and}\,\, V'' >0 \,\,\text{or}\,\, V''<0.
\end{align*}
\begin{lemma}\label{lemma: omconst}
Let $\nu k^2\leq \|V'\|_{L^\infty}$ and $(\om,\psi)$ be the solution to \eqref{equ: om const}. Then there exist constants $\nu_0>0$ and $\epsilon_1>0$ such that, for $\nu\in(0,\nu_0]$, $\epsilon\in[0,\epsilon_{1}]$, we have
\begin{align}\label{esti: omconst}
&|k|^2\bn{e^{\epsilon \nu^{\frac13}t} u}^2_{L^2L^2}+\nu^{\frac12} | k| \|e^{\epsilon \nu^{\frac13}t}\om\|^2_{L^2L^2}+\nu^{\frac13}|k|^{\frac23}\|e^{\epsilon \nu^{\frac13}t}\rho^{\frac12}_k\om\|^2_{L^2L^2}\\
&\lesssim E^{\mathrm{in}}+ \nu^{-1}\|e^{\epsilon \nu^{\frac13} t}(f_1, f_2)\|^2_{L^2 L^2}+\nu^{-\frac13}|k|^{-\frac23}\|e^{\epsilon \nu^{\frac13} t}f_3\|^2_{L^2 L^2}+\nu^{-\frac{1}{6}}|k|^{-\frac73}\|e^{\epsilon \nu^{\frac13} t}f_4\|^2_{L^2 H_{k}^{1}}.\notag
\end{align}
\end{lemma}
The estimates for $|k|^2\bn{e^{\epsilon \nu^{\frac13}t} u}^2_{L^2L^2}$ have been
established in \cite[Propositions 6.1 and 6.2]{Chen-Wei-Zhang1} that
\begin{align*}
|k|^2\| e^{\epsilon \nu^{\frac13} t} u\|^2_{L^2L^2}
\lesssim & E^{\mathrm{in}}+\nu^{-1}\|e^{\epsilon \nu^{\frac13}t}(f_1,f_2)\|^2_{L^2L^2}\\
&+\nu^{-\frac13}|k|^{-\frac23}\|e^{\epsilon \nu^{\frac13}t}f_3\|^2_{L^2 L^2} +|k|^{-2}\|e^{\epsilon \nu^{\frac13}t}f_4\|^2_{L^2 H_{k}^{1}}.
\end{align*}
We are left to bound $\nu^{\frac12} | k| \|e^{\epsilon \nu^{\frac13}t}\om\|^2_{L^2L^2}+\nu^{\frac13}|k|^{\frac23}\|e^{\epsilon \nu^{\frac13}t}\rho^{\frac12}_k\om\|^2_{L^2L^2}$, whose proof has been divided into the following Lemmas \ref{Prop: inhomo} and \ref{Prop homo pro}.

We first decompose the solution to \eqref{equ: om const} into the inhomogeneous and homogeneous parts:
$$\om=\om_{I}+\om_{H},$$
where the inhomogeneous part $\om_{I}$ satisfies
\begin{equation}\label{equ: omI}
\left\{
\begin{aligned}
&\pa_t \om_{I}-\nu(\pa^2_y- k^2)\om_{I}+i kV\om_{I}-ik V''\psi_{I}=-i k f_1-\pa_y f_2-f_3-f_4, \\
&\om_{I}=(\pa^2_y- k^2)\psi_{I},\quad \psi_{I}(0)=\psi_{I}(1)=\psi'_{I}(0)=\psi'_{I}(1)=0,\\
&\om_{I}(0,y)=0,
\end{aligned}
\right.
\end{equation}
and the homogeneous part $\om_{H}$ satisfies
\begin{equation}\label{equ: omH}
\left\{
\begin{aligned}
&\pa_t \om_{H}-\nu(\pa^2_y- k^2)\om_{H}+i kV\om_{H}-ik V''\psi_{H}=0,\\
&\om_{H}=(\pa^2_y- k^2)\psi_{H},\quad \psi_{H}(0)=\psi_{H}(1)=\psi'_{H}(0)=\psi'_{H}(1)=0,\\
&\om_{H}(0,y)=\om_{k}^{\mathrm{in}}(y).
\end{aligned}
\right.
\end{equation}

\begin{lemma}\label{Prop: inhomo}
Let $\nu k^2\leq \|V'\|_{L^\infty}$ and $(\om_{I},\psi_{I})$ be the solution to \eqref{equ: omI}.
Then
there exist constants $\nu_0>0$ and $\epsilon_1>0$ such that, for $\nu\in(0,\nu_0]$, $\epsilon\in[0,\epsilon_{1}]$, we have
\begin{align}\label{esti: inhomo}
&\nu^{\frac12} | k| \|e^{\epsilon \nu^{\frac13}t}\om_{I}\|^2_{L^2L^2}+\nu^{\frac13}|k|^{\frac23}\|e^{\epsilon \nu^{\frac13}t}\rho^{\frac12}_k\om_{I}\|^2_{L^2L^2}\\
\lesssim& \nu^{-1}\|e^{\epsilon \nu^{\frac13}t}(f_1, f_2)\|^2_{L^2 L^2}+\nu^{-\frac13}|k|^{-\frac23}\|e^{\epsilon \nu^{\frac13}t}f_3\|^2_{L^2 L^2} +\nu^{-\frac{1}{6}}|k|^{-\frac73}\|e^{\epsilon \nu^{\frac13}t}f_4\|^2_{L^2 H_{k}^{1}}.\notag
\end{align}
\end{lemma}
\begin{proof}
First, we set
\begin{align*}
&w_I(\la, k,y)=:\int^{+\infty}_0 e^{\epsilon \nu^{\frac13}t} \om_{I}(t, k,y) e^{-it\la} dt,\\
&F_{j}(\la, k,y)=:\int^{+\infty}_0 e^{\epsilon \nu^{\frac13}t}f_{j}(t, k,y)e^{-it\la} dt,\quad  j=1,2,3,4.
\end{align*}
Then $w_{I}$ satisfies the OS equation with no-slip boundary condition. Therefore, we can appeal to the resolvent estimates in Proposition \ref{lemma:non-slip boundary,resolvent} to obtain
\begin{align*}
\nu^{\frac12}|k|\|w_{I}\|^2_{L^2}+\nu^{\frac13}|k|^{\frac23}\|\rho^{\frac12}_k w_{I}\|^2_{L^2}\lesssim \nu^{-1}\|(F_1, F_2)\|^2_{L^2}+\nu^{-\frac13}|k|^{-\frac23}\|F_3\|^2_{ L^2}+\nu^{-\frac16}|k|^{-\frac73}\| F_4\|^2_{H_{k}^{1}},
\end{align*}
which, along with the Plancherel's formula, implies that
\begin{align*}
&\nu^{\frac12} | k| \|e^{\epsilon \nu^{\frac13}t}\om_{I}\|^2_{L^2L^2}+\nu^{\frac13}|k|^{\frac23}\|e^{\epsilon \nu^{\frac13}t}\rho^{\frac12}_k\om_{I}\|^2_{L^2L^2}\\
&\lesssim \nu^{-1}\|e^{\epsilon \nu^{\frac13}t}(f_1, f_2)\|^2_{L^2 L^2}+\nu^{-\frac13}|k|^{-\frac23}\|e^{\epsilon \nu^{\frac13}t}f_3\|^2_{L^2 L^2} +\nu^{-\frac{1}{6}}|k|^{-\frac73}\|e^{\epsilon \nu^{\frac13}t}f_4\|^2_{L^2 H_{k}^{1}}.
\end{align*}
This completes the proof of Lemma \ref{Prop: inhomo}.
\end{proof}

Next, we get into the analysis of the homogeneous equation \eqref{equ: omH}.
\begin{lemma}\label{Prop homo pro}
Let $\nu k^2\leq \|V'\|_{L^\infty}$ and
$(\om_{H},\psi_{H})$ be the solution to \eqref{equ: omH}. Then
there exist constants $\nu_0>0$ and $\epsilon_1>0$ such that, for $\nu\in(0,\nu_0]$, $\epsilon\in[0,\epsilon_{1}]$, we have
\begin{equation}\label{homo om}
\begin{aligned}
\nu^{\frac12}|k|\|e^{\epsilon \nu^{\frac13}t} \om_{H}\|^2_{L^2L^2}+\nu^{\frac13}|k|^{\frac23}\|e^{\epsilon \nu^{\frac13}t}\rho^{\frac12}_k\om_{H}\|^2_{L^2L^2}
\lesssim E^{\mathrm{in}}.
\end{aligned}
\end{equation}
\end{lemma}
\begin{proof}
Let $\om_{H,0}$ solve the linearized Euler equation
\begin{align*}
\pa_t \om_{H,0}+ik V \om_{H,0}-ikV''(\pa^2_y-k^2)^{-1}\om_{H,0}=0, \quad \om_{H,0}(0)=\om_{k}^{\mathrm{in}}.
\end{align*}
To incorporate the effect of enhanced dissipation, we then set
\begin{align*}
&\om_{H,1}(t, k,y)=e^{-\nu k^2 ((V')^2t^3/3+t)}\om_{H,0}(t,k,y),
\end{align*}
which satisfies
\begin{align}\label{equ: omH1}
&\pa_t\om_{H,1}-\nu(\pa^2_y- k^2)\om_{H,1}+i k V\om_{H,1}-ikV''\psi_{H,1}\\
&\qquad=-\nu (\pa_y\om_{1,1}-ikt V'' \om_{H,1}-iktV'\om_{1,1})+ikV''\psi_{1,1},\notag
\end{align}
where $\psi_{1,1}=\psi_{H,1}-e^{-\nu k^2 ((V')^2t^3/3+t)}\psi_{H,0}$ and $\om_{1,1}=\pa_y \om_{H,1}+iktV'\om_{H,1}$.

We also introduce the notations
\begin{align*}
u_{H,j}=(\pa_y, -i k)\psi_{H,j}, \quad \psi_{H,j}=(\pa^2_y- k^2)^{-1}\om_{H,j}, \quad j=0,1,2,3.
\end{align*}
Then we decompose $\om_{H}$ into three parts
\begin{align*}
\om_{H}=\om_{H,1}+\om_{H,2}+\om_{H,3},
\end{align*}
where $\om_{H,2}$ satisfies the inhomogeneous NS equation with  no-slip boundary condition
\begin{equation}\left\{\begin{aligned}
&\pa_t\om_{H,2}-\nu(\pa^2_y- k^2)\om_{H,2}+i kV\om_{H,2}-ikV''\psi_{H,2}\\
&\qquad \qquad =\nu (\pa_y\om_{1,1}-ikt V'' \om_{H,1}-iktV'\om_{1,1})-ikV''\psi_{1,1},\\
&\om_{H,2}|_{t=0}=0,\qquad \lan \om_{H,2}, \sinh (ky)\ran =\lan \om_{H,2}, \sinh (k(1-y))\ran=0,
\end{aligned}\right.\label{equ: omH2}\end{equation}
and $\om_{H,3}$ solves
\begin{equation}\left\{
\begin{aligned}
&\pa_t\om_{H,3}-\nu(\pa^2_y- k^2)\om_{H,3}+i k V\om_{H,3}-ik V''\psi_{H,3}=0, \\
&\om_{H,3}|_{t=0}=0, \qquad \lan \om_{H,3}+\om_{H,1}, \sinh (ky)\ran =\lan \om_{H,3}+ \om_{H,1}, \sinh (k(1-y))\ran=0.\label{equ: omH3}
\end{aligned}
\right.
\end{equation}

\textbf{Step 1.} Estimates of $\om_{H,1}$.\smallskip

The following estimates for $\om_{H,1}$ have been established in \cite[Proposition 6.2]{Chen-Wei-Zhang1} that
\begin{align}
&\nu^{\frac13} |k|^{\frac23} \|e^{\epsilon \nu^{\frac13}t}\om_{H,1}\|^2_{L^2L^2}\lesssim \|\om^{\mathrm{in}}_{k}\|^2_{L^2},\label{omH1L2L2}\\
&\nu \|e^{\epsilon \nu^{\frac13}t}\om_{1,1}\|^2_{L^2L^2}\lesssim \nu^{\frac23}| k|^{-\frac23}\| \om^{\mathrm{in}}_{k}\|^2_{H^1}\label{payom1HL2L2},\\
&\nu^{\frac53}|k|^{\frac43}\lrs{\|e^{\epsilon \nu^{\frac13}t}t\om_{H,1}\|^2_{L^2L^2}+\|e^{\epsilon \nu^{\frac13}t}t\om_{1,1}\|^2_{L^2L^2}}\lesssim \nu^{\frac23}|k|^{-\frac23}\|\om^{\mathrm{in}}_{k}\|^2_{H^{1}},\label{esti: om1HL2}\\
&\nu^{-\frac16}|k|^{-\frac{1}{3}}\|e^{\epsilon \nu^{\frac13}t}   \psi_{1,1} \|^2_{L^2H_{k}^{1}}\lesssim \nu^{\frac12}|k|^{-1}\|\om^{\mathrm{in}}_{k}\|^2_{H^{1}}.\label{u1HL2}
\end{align}

\textbf{Step 2.} Estimates of $\om_{H,2}$.\smallskip

Applying Lemma \ref{Prop: inhomo} with
\begin{align*}
f_2=-\nu \om_{1,1},\quad f_3=-i\nu kt(V''\om_{H,1}+V'\om_{1,1}),\quad f_4=-ikV''\psi_{1,1},
\end{align*}
we have
\begin{align}\label{esti: om2HL2L2}
&\nu^{\frac12} | k|\|e^{\epsilon \nu^{\frac13}t}\om_{H,2}\|^2_{L^2L^2}+\nu^{\frac13}|k|^{\frac23}\|e^{\epsilon \nu^{\frac13}t}\rho^{\frac12}_k\om_{H,2}\|^2_{L^2L^2}\\
&\lesssim \nu\|e^{\epsilon \nu^{\frac13}t} \om_{1,1}\|^2_{L^2L^2}+\nu^{\frac53}|k|^{\frac43}\| e^{\epsilon \nu^{\frac13}t}t(V''\om_{H,1}+V'\om_{1,1}) \|^2_{L^2L^2}\notag\\
&\quad+\nu^{-\frac16}|k|^{-\frac{7}{3}} \|e^{\epsilon \nu^{\frac13}t}( ik V'' \psi_{1,1}) \|^2_{L^2H_{k}^{1}}.\notag
\end{align}

Inserting \eqref{payom1HL2L2}--\eqref{u1HL2} into \eqref{esti: om2HL2L2}, we obtain
\begin{equation}\label{esti:omH2}
\begin{aligned}
\nu^{\frac12} | k| \|e^{\epsilon \nu^{\frac13}t}\om_{H,2}\|^2_{L^2L^2}+\nu^{\frac13}|k|^{\frac23}\|e^{\epsilon \nu^{\frac13}t}\rho^{\frac12}_k\om_{H,2}\|^2_{L^2L^2}\lesssim |k|^{-2}\|\om^{\mathrm{in}}_{k}\|^2_{H^{1}}.
\end{aligned}
\end{equation}

\textbf{Step 3.} Estimates of $\om_{H,3}$.\smallskip

Taking the Fourier transform in the time-variable $t$, we have
\begin{align*}
&w_{3}(\la, k,y)=\int^{+\infty}_0\om_{H,3}(t, k,y) e^{-it\la+\epsilon \nu^{\frac13}t}dt,\quad
\phi_{3}(\la,k,y)=\int^{+\infty}_{0}\psi_{H,3}e^{-it\la+\epsilon \nu^{\frac13}t}dt,
\end{align*}
and
\begin{align*}
(i\la-\epsilon \nu^{\frac13}-\nu (\pa^2_y- k^2)+i k V) w_{3}(\la, k,y)-ik V'' \phi_{3}=0.
\end{align*}
Thus, we get
\begin{align*}
w_{3}=-c_1(\la)w_{3,1}-c_2(\la)w_{3,2},
\end{align*}
and
\begin{align*}
c_1(\la)=-\int^{1}_{0} \frac{\sinh (k(1-y))}{\sinh k} w_{3} dy,\qquad c_2(\la)=-\int^{1}_{0} \frac{\sinh (ky)}{\sinh k} w_{3} dy, 
\end{align*}
where $w_{3,1}$, $w_{3,2}$ are the boundary layer correctors defined by the  homogeneous OS equations \eqref{equ: psi1} and \eqref{equ: psi2} with $\la$ replaced by $-\la/k-i\epsilon \nu^{\frac13}/k$.

By \eqref{w1L2}--\eqref{weightL2}, we infer
\begin{align*}
\|w_{H,3}(\la,k,y)\|_{L^2_y}\leq &|c_1(\la)|\|w_{3,1}\|_{L^2_y}+|c_2(\la)|\|w_{3,2}\|_{L^2_y}\\
\lesssim & \nu^{-\frac14}\big(|c_1(\la)|(1+|\la+kV(0)|)^{\frac14}+|c_2(\la)|(1+|\la+kV(1)|)^{\frac14}\big),\\
\|\rho^{\frac12}_k w_{H,3}(\la,k,y)\|_{L^2_y}\leq & |c_1(\la)|\|\rho^{\frac12}_kw_{3,1}\|_{L^2_y}+|c_2(\la)|\|\rho^{\frac12}_kw_{3,2}\|_{L^2_y}
\lesssim \nu^{-\frac16}|k|^{\frac16} (|c_1(\la)|+|c_2(\la)|),
\end{align*}
which implies that
\begin{align}\label{equ:wH3L2}
&\nu^{\frac14}|k|^{\frac12}\|e^{\epsilon \nu^{\frac13}t}\om_{H,3}\|_{L^2L^2}+\nu^{\frac16}|k|^{\frac13}\|e^{\epsilon \nu^{\frac13}t}\rho^{\frac12}_k \om_{H,3}\|_{L^2L^2}\\
&\lesssim  |k|^{\frac12} (\|(1+|\la+kV(0)|)^{\frac14}c_1\|_{L^2}+\|(1+|\la+kV(1)|)^{\frac14}c_2\|_{L^2}).\notag
\end{align}
The coefficients $c_{1}$ and $c_{2}$ have been bounded in \cite[(7.21)]{Chen-Wei-Zhang1} as
\begin{align}\label{equ:coefficient,c1,c2}
\|(1+|\la+kV(0)|)c_1\|_{L^2}+\|(1+|\la+kV(1)|)c_2\|_{L^2}\lesssim  |k|^{-\frac32}\|\om^{\mathrm{in}}_{k}\|_{H^{1}}.
\end{align}
Putting \eqref{equ:coefficient,c1,c2} into \eqref{equ:wH3L2}, we arrive at
\begin{equation}\label{esti:omH3}
\begin{aligned}
\nu^{\frac12} | k| \|e^{\epsilon \nu^{\frac13}t}\om_{H,3}\|^2_{L^2L^2}+\nu^{\frac13}|k|^{\frac23}\|e^{\epsilon \nu^{\frac13}t}\rho^{\frac12}_k \om_{H,3}\|^2_{L^2L^2}\lesssim |k|^{-2}\|\om^{\mathrm{in}}_{k}\|^2_{H^{1}}.
\end{aligned}
\end{equation}

Combining \eqref{omH1L2L2}, \eqref{esti:omH2}  with \eqref{esti:omH3}, we complete the proof of Lemma \ref{Prop homo pro}.
\end{proof}

\subsection{Space-time estimates for low frequencies: heat case}\label{sec:Space-time estimates for low frequencies: heat flow case}

For the subsequent application of the time-frozen method, we require uniform control over the background shear flow $U(t,y)$, which is governed by the heat equation
\begin{equation}
\left\{
\begin{aligned}
&\pa_{t}U=\nu\pa_{y}^{2}U,\\
&U(0,y)=U^{\mathrm{in}}(y),\quad U(t,0)=U^{\mathrm{in}}(0),\quad U(t,1)=U^{\mathrm{in}}(1).
\end{aligned}
\right.
\end{equation}
 The necessary uniform estimates are established in the lemma below.
\begin{lemma}
Assume that the initial datum $U^{\mathrm{in}}$ satisfies the $(M)$ condition \eqref{equ:condition,shear,flow,M}. Then the following estimates hold.
\begin{enumerate}[$(1)$]
\item
The uniform lower and upper bounds for $\pa_{y}U$:
\begin{align}
0<c_{0}\leq \inf_{t\in[0,\infty)}\inf_{y\in[0,1]}\pa_{y}U\leq \n{\pa_{y}U}_{L_{t}^{\infty}L_{y}^{\infty}}\leq C_{0}.
\end{align}
\item Instantaneous positivity property for $\pa_{y}^{2}U$:
\begin{align}
&\text{if $\pa_{y}^{2}U^{\mathrm{in}}\geq 0$, then $\pa_{y}^{2}U\equiv0$, or $\pa_{y}^{2}U(t,y)>0$, $\f t>0, y\in (0,1)$}, \\
&\text{if $\pa_{y}^{2}U^{\mathrm{in}}\leq 0$, then $\pa_{y}^{2}U\equiv0$, or $\pa_{y}^{2}U(t,y)<0$, $\f t>0, y\in (0,1)$}.
\end{align}
\item Regularity estimates:
\begin{align}
\n{U(t,y)}_{L_{t}^{\infty}H^{4}}\lesssim& 1,\label{equ:U,H4,estimate}\\
\|U(t,y)-U(s,y)\|_{ L^\infty_{y}}\lesssim & \nu(t-s)\|U^{\mathrm{in}}\|_{H^4},\label{esti: Ut-Us Linfty}\\
\|\pa^2_y U(t,y)-\pa^2_y U(s,y)\|_{L^2_{y}}\lesssim & \nu(t-s)\|U^{in}\|_{H^4}.\label{esti: pay3Ut-UsLinfty}
\end{align}
\end{enumerate}
\end{lemma}
\begin{proof}
For $(1)$, we
notice that $\pa_{y}U$ satisfies the heat equation with the Neumann boundary condition. The solution can be expressed using the Neumann heat kernel
\begin{align}
\pa_{y}U(t,y)=\int_{0}^{1}K(t,y,z)\pa_{z}U^{\mathrm{in}}(z)dz,
\end{align}
where the Neumann heat kernel is positive and satisfies the conservation property
\begin{align}
K(t,y,z)>0,\quad \int_{0}^{1}K(t,y,z)dz=1.
\end{align}
Hence, together with the lower and upper bounds for the initial datum, we derive
\begin{align}
0<c_{0}\leq \inf_{z\in(0,1)}\pa_{z}U^{\mathrm{in}}(z)\leq \pa_{y}U(t,y)\leq \sup_{z\in(0,1)}\pa_{z}U^{\mathrm{in}}(z)\leq C_{0}.
\end{align}

For $(2)$, due to that $\pa_{y}^{2}U$ satisfies the heat equation with the Dirichlet boundary condition, the instantaneous positivity property follows directly from the strong maximum principle.

For $(3)$, we set
$$f(t,y)=:U(t,y)-\lrs{U^{\mathrm{in}}(1)-U^{\mathrm{in}}(0)}y-U^{\mathrm{in}}(0),$$
which solves the heat equation with the Dirichlet boundary condition.
Applying the $H^{2}$ estimate \eqref{equ:heat,dirichlet,H2} for the Dirichlet heat equation, we have
\begin{align}
\n{f}_{L_{t}^{\infty}H^{2}}\leq \n{f(0)}_{H^{2}},
\end{align}
which implies that
\begin{align}
\n{U}_{L_{t}^{\infty}H^{2}}\lesssim 1.\label{equ:U,H2,estimate}
\end{align}
Similarly, with the initial condition $\pa_{y}^{2}U^{\mathrm{in}}\in H^{2}\cap H_{0}^{1}$, we also obtain
\begin{align}
&\|\pa_y^{2}U\|_{L_{t}^{\infty}H^2}\leq \|\pa_y^{2} U^{\mathrm{in}}\|_{H^2}. \label{pay2U-yH2}
\end{align}
Putting together \eqref{equ:U,H2,estimate} and \eqref{pay2U-yH2}, we arrive at \eqref{equ:U,H4,estimate}.

Due to the initial condition $U^{\mathrm{in}}\in H^{2}\cap H_{0}^{1}$, we rewrite
\begin{align*}
U(t,y)-U(s,y)=\int_{s}^{t}\pa_{\tau} U(\tau,y) d\tau=\nu \int^{t}_{s}\pa_y^2 U(\tau,y) d\tau.
\end{align*}
Then by the interpolation inequality and $H^{2}$ estimate \eqref{pay2U-yH2}, we derive
\begin{align*}
\|U(t,y)-U(s,y)\|_{L^\infty_{y}}\leq& \nu\int^{t}_{s}\|\pa_y^2 U(\tau,y)\|_{L^\infty_{y}} d\tau
\lesssim \nu\int^{t}_{s}\|\pa^2_y U(\tau,y)\|^{\frac12}_{L^2} \|\pa^2_y U(\tau,y)\|^{\frac12}_{\dot{H}^1}dy\\
\leq& \nu(t-s)\|U^{\mathrm{in}}\|_{H^4}.
\end{align*}
Similarly, for the second derivative, we also have
\begin{align*}
\bn{\pa^2_y U(t)-\pa^2_yU(s)}_{ L^2_{y}}\leq& \nu\int^{t}_{s}\bn{\pa_y^4 U(\tau,y)}_{L^2_{y}} d\tau\leq \nu(t-s)\|U^{\mathrm{in}}\|_{H^4}.
\end{align*}
\end{proof}

Now, we consider the solution to the equation \eqref{equ: om} for $\nu k^2\leq \n{\pa_{y}U}_{L_{t}^{\infty}L_{y}^{\infty}}$ and establish the following space-time estimates.
\begin{proposition}\label{lemma: uL2+omL2}
Let $\nu k^2\leq \n{\pa_{y}U}_{L_{t}^{\infty}L_{y}^{\infty}}$, and $(\om,\psi)$ be the solution to \eqref{equ: om}. Then
there exist constants $\nu_0>0$ and $\epsilon_{0}>0$ such that, for $\nu\in(0,\nu_0]$, $\epsilon\in(0,\epsilon_{0}]$, it holds that
\begin{align}\label{equ:low frequency,space-time,flow}
&|k|^2\|e^{\epsilon \nu^{\frac13}t}u\|^2_{L^2 L^2}+\nu^{\frac12} |k|\|e^{\epsilon \nu^{\frac13}t}\om\|^2_{L^2 L^2}+\nu^{\frac13}|k|^{\frac23}\|e^{\epsilon \nu^{\frac13}t}\rho^{\frac12}_k \om\|^2_{L^2L^2}\\
&\lesssim E^{\mathrm{in}}+\nu^{-1}\|e^{\epsilon \nu^{\frac13}t}(f_1,f_2)\|^2_{L^2 L^2}.\notag
\end{align}
\end{proposition}
\begin{proof}
It suffices to prove \eqref{equ:low frequency,space-time,flow} for any fixed time $T$. Without loss of generality, we take
$$T=t_{N+1},\quad t_{j}=j\nu^{-\frac{1}{3}}.$$
We decompose the solution
$\omega$ of \eqref{equ: om} into $\om=\sum_{j\geq 0} \om_{j}$, where for each $j\geq 0$, the component $\omega_{j}$ satisfies
\begin{equation}
\left\{
\begin{aligned}
&\pa_{t}\om_{j}-\nu (\pa^2_y-k^2)\om_{j}+ik U(t_{j+1},y)\om_{j}-ik \pa^2_y U (t_{j+1},y)\psi_{j}
=[ik f_{1}+\pa_{y} f_2+G_{j}]\chi_{I_{j}},\\
&\om_{j}=(\pa^2_y- k^2)\psi_{j},\quad \psi_{j}(0)=\psi_{j}(1)=\psi'_{j}(0)=\psi'_{j}(1)=0,\\
&\om_{j}(t_{j},y)=1_{\lr{j=0}}\om^{\mathrm{in}}(y),
\end{aligned}
\right.
\end{equation}
with $I_{j}=:[t_{j}, t_{j+1})$ and
\begin{align*}
&G_{j}=\sum_{j'=0}^{j}(F_{j',1}-F_{j',2}),\\
&F_{j',1}:=ik (U(t_{j'+1},y)-U(t,y))\om_{j'}, \quad F_{j',2}:=ik (\pa^2_y U(t_{j'+1},y)-\pa^2_y U(t,y))\psi_{j'}.
\end{align*}
For simplicity,
we define the norms
\begin{align}
&\|f\|_{Y}=|k|\|(\pa_y ,k)(\pa^2_y-k^2)^{-1}f\|_{L^2}
+\nu^{\frac14} |k|^{\frac12}\|f\|_{L^2}+\nu^{\frac16} |k|^{\frac13}\|\rho^{\frac12}_kf\|_{L^2},\label{def: omjY}\\
&\|f\|_{X(a,b)}=\|f\|_{L^2(a,b;Y)}.\label{def: omjXj}
\end{align}

For $0 \leq j\leq N$, noting that $\epsilon \nu^{\frac13}(t-t_{j})\leq 1$ for $t\in \mathrm{I}_{j}$, we use Lemma \ref{lemma: omconst} with $f_3=G_{j}\chi_{I_{j}}$ and $f_{4}=0$ to deduce
\begin{align}\label{esti: omjwithjgeq1}
&\|e^{\epsilon_{1} \nu^{\frac13}(t-t_{j})}\om_{j}\|^2_{X(t_{j},T)}\\
&\lesssim E^{\mathrm{in}}1_{\lr{j=0}}+\nu^{-1}\|e^{\epsilon_{1} \nu^{\frac13}(t-t_{j})}(f_1,f_2)\|^2_{L^2(\mathrm{I}_{j}; L^2)}
+\nu^{-\frac13}|k|^{-\frac23}\|e^{\epsilon_{1} \nu^{\frac13}(t-t_{j})}G_{j}\|^2_{L^2(\mathrm{I}_{j}; L^2)}\notag\\
&\lesssim E^{\mathrm{in}}1_{\lr{j=0}}+ \nu^{-1}\|(f_1,f_2)\|^2_{L^2(\mathrm{I}_{j}; L^2)}+\nu^{-\frac13}|k|^{-\frac23}\|G_{j}\|^2_{L^2(\mathrm{I}_{j}; L^2)}.\notag
\end{align}

It follows from \eqref{esti: Ut-Us Linfty} and \eqref{esti: pay3Ut-UsLinfty} that
\begin{align*}
&\n{U(t_{j'+1},y)-U(t,y)}_{L^{\infty}(\mathrm{I}_{j}; L^{\infty})}+\n{\pa^2_y U(t_{j'+1},y)-\pa^2_y U(t,y)}_{L^{\infty}(\mathrm{I}_{j}; L^{2})}\\
&\lesssim \nu \|t-t_{j'+1}\|_{L_{t}^\infty(\mathrm{I}_{j})}\|U^{in}\|_{H^4}\lesssim \nu^{\frac23}(j-j'+1).
\end{align*}
Combining this with the bound $\n{\psi_{j'}}_{L^{\infty}}\lesssim |k|^{-\frac{3}{2}}\n{w_{j'}}_{L^{2}}\leq \n{w_{j'}}_{L^{2}}$, we derive
\begin{align}\label{esti:F1+F2}
\nu^{-\frac16}|k|^{-\frac13}\|G_{j}\|_{L^2(\mathrm{I}_{j}; L^2)}
\leq &\nu^{-\frac16}|k|^{-\frac13}\sum_{j'=0}^{j}(\|F_{j',1}\|_{L^2(\mathrm{I}_{j}; L^2)}+\|F_{j',2}\|_{L^2(\mathrm{I}_{j}; L^2)} )\\
\leq &\nu^{-\frac16}|k|^{\frac23}\sum_{j'=0}^{j}\Big(\|U(t_{j'+1},y)-U(t,y)\|_{L^{\infty}(\mathrm{I}_{j}; L^{\infty})}\n{\om_{j'}}_{L^2(\mathrm{I}_{j};L^2)}\notag\\
&+\n{\pa^2_y U(t_{j'+1},y)-\pa^2_y U(t,y)}_{L^{\infty}(\mathrm{I}_{j}; L^{2})}\|\psi_{j'}\|_{L^2(\mathrm{I}_{j}; L^\infty)}\Big)\notag\\
\lesssim & \nu^{\frac{1}{2}}|k|^{\frac23}\sum_{j'=0}^{j}(j-j'+1)\|\om_{j'}\|_{L^2(\mathrm{I}_{j}; L^2)}.\notag
\end{align}

Since $\nu^{\frac13}(t-t_{j'})\geq j-j'$ for $t\in \mathrm{I}_{j}$ and $\nu k^2\lesssim 1$, it follows that
\begin{align*}
\nu^{\frac{1}{2}}|k|^{\frac23}\|\om_{j'}\|_{L^2(\mathrm{I}_{j}; L^2)}
\lesssim  \nu^{\frac{5}{12}}|k|^{\frac12}e^{-\epsilon_{1} (j-j')}\n{e^{\epsilon_{1} \nu^{\frac{1}{3}}(t-t_{j'})} \om_{j'}}_{L^2(\mathrm{I}_{j};L^2)}.
\end{align*}
Then by \eqref{esti:F1+F2}, the definition \eqref{def: omjXj} of the norm $X$, and Cauchy-Schwarz inequality, we derive
\begin{align}\label{esti:G1+G2}
\nu^{-\frac13}|k|^{-\frac23}\|G_{j}\|^2_{L^2(\mathrm{I}_{j}; L^2)}
 \lesssim & \nu^{\frac{1}{3}}  \lrs{\sum_{j'=0}^{j}(j-j'+1)  e^{-\epsilon_{1} (j-j') }\|e^{\epsilon_{1} \nu^{\frac{1}{3}}(t-t_{j'})}\om_{j'}\|_{X(\mathrm{I}_{j})}}^2\\
 \lesssim&  \nu^{\frac{1}{3}}  \sum_{j'=0}^{j} e^{-\epsilon_{1} (j-j') }\|e^{\epsilon_{1} \nu^{\frac{1}{3}}(t-t_{j'})}\om_{j'}\|^2_{X(\mathrm{I}_{j})}. \notag
\end{align}

Inserting \eqref{esti:G1+G2} into \eqref{esti: omjwithjgeq1}, we get
\begin{align}\label{esti: omjXj}
&\|e^{\epsilon_{1} \nu^{\frac13}(t-t_{j})}\om_{j}\|^2_{X(t_{j},T)}\\
&\lesssim  E^{\mathrm{in}}1_{\lr{j=0}}+ \nu^{-1}\|(f_1,f_2)\|^2_{L^2(\mathrm{I}_{j}; L^2)}+\nu^{\frac{1}{3}}\sum_{j'=0}^{j}  e^{-\epsilon_{1}(j-j')} \|e^{\epsilon_{1} \nu^{\frac{1}{3}}(t-t_{j'})}\om_{j'}\|^2_{X(\mathrm{I}_{j})}. \notag
\end{align}

Now we set
\begin{align*}
D_{j}=\sum_{j'=0}^{j}  e^{-\epsilon_{1} (j-j')} \|e^{\epsilon_{1} \nu^{\frac{1}{3}}(t-t_{j'})}\om_{j'}\|^2_{X(t_{j'}, t_{j+1})}.
\end{align*}
Then we get by \eqref{esti: omjXj} that
\begin{align*}
D_{j}
\lesssim
& \sum_{j'=0}^{j}  e^{-\epsilon_{1} (j-j')} \Big[E^{\mathrm{in}}1_{\lr{j'=0}}+
\nu^{-1}\|(f_1,f_2)\|^2_{L^2(\mathrm{I}_{j'}; L^2)}+\nu^{\frac{1}{3}}D_{j'}\Big].
\end{align*}

Choosing $\epsilon\leq \frac{\epsilon_{1}}{4}$ and applying Minkowski's inequality, we get
\begin{align*}
\sum_{j=0}^{N}e^{2\epsilon j}D_{j}
\lesssim
& \sum_{j=0}^{N}e^{2\epsilon j}\sum_{j'=0}^{j}  e^{-\epsilon_{1} (j-j')} \Big[E^{\mathrm{in}}1_{\lr{j'=0}}+
\nu^{-1}\|(f_1,f_2)\|^2_{L^2(\mathrm{I}_{j'}; L^2)}+\nu^{\frac{1}{3}}D_{j'}\Big]\\
=&\sum_{j=0}^{N}\sum_{j'=0}^{j}  e^{-(\epsilon_{1}-2\epsilon )(j-j')}e^{2\epsilon j' } \Big[E^{\mathrm{in}}1_{\lr{j'=0}}+
\nu^{-1}\|(f_1,f_2)\|^2_{L^2(\mathrm{I}_{j'}; L^2)}+\nu^{\frac{1}{3}}D_{j'}\Big]\\
\lesssim& \sum_{j'=0}^{N} e^{2\epsilon j' } \Big[E^{\mathrm{in}}1_{\lr{j'=0}}+
\nu^{-1}\|(f_1,f_2)\|^2_{L^2(\mathrm{I}_{j'}; L^2)}+\nu^{\frac{1}{3}}D_{j'}\Big]\\
\leq& E^{\mathrm{in}} + \nu^{-1}\n{e^{\epsilon  \nu^{\frac13} t}(f_1,f_2)}^2_{L^2(0, T; L^2)}+\nu^{\frac13} \sum_{j=0}^{N} e^{2\epsilon  j }D_{j},
\end{align*}
which implies that
\begin{align}\label{esti: sumDj}
\sum_{j=0}^{N}e^{2\epsilon  j}D_{j}
 \lesssim E^{\mathrm{in}} + \nu^{-1}\n{e^{\epsilon  \nu^{\frac13} t}(f_1,f_2)}^2_{L^2(0, T; L^2)}.
\end{align}

Finally, since $(0,T)=\sum_{j=0}^{N}\mathrm{I}_{j}$ and $\om=\sum_{j'=0}^{j}\om_{j'}$ for $t\in \mathrm{I}_{j}$, we have
\begin{align}\label{esti:omX0T}
\|e^{\epsilon \nu^{\frac13} t}\om\|^2_{X(0,T)}=&\sum_{j=0}^{N}\int_{\mathrm{I}_{j}}\|e^{\epsilon \nu^{\frac13}t}\om\|^2_{Y} dt
\lesssim\sum_{j=0}^{N}e^{2\epsilon j}\int_{\mathrm{I}_{j}}\bbn{\sum_{j'=0}^{j}\om_{j'}}^2_{Y} dt.
\end{align}
Using H\"older's inequality and the fact that $\nu^{\frac13}(t-t_{j'})\geq j-j'$ for $t\in \mathrm{I}_{j}$, we get
\begin{align}\label{equ:Ij,sum,j,Y}
\int_{\mathrm{I}_{j}}\lrs{\sum_{j'=0}^{j}\n{\om_{j'}}_{Y}}^2 dt\leq&\int_{\mathrm{I}_{j}}\sum_{j'=0}^{j}e^{-\epsilon_{1} \nu^{\frac13}(t-t_{j'})}\n{e^{\epsilon_{1} \nu^{\frac13}(t-t_{j'})}\om_{j'}}^2_{Y}\sum_{j'=0}^{j}e^{-\epsilon_{1} \nu^{\frac13}(t-t_{j'})} dt\\
\lesssim& \int_{\mathrm{I}_{j}}\sum_{j'=0}^{j} e^{-\epsilon_{1} (j-j')} \n{e^{\epsilon_{1} \nu^{\frac13}(t-t_{j'})}\om_{j'}}^2_{Y} dt\lesssim D_{j}.\notag
\end{align}
Putting together \eqref{esti: sumDj}, \eqref{esti:omX0T}, and \eqref{equ:Ij,sum,j,Y}, we arrive at
\begin{align*}
\|e^{\epsilon \nu^{\frac13} t}\om\|^2_{X(0,T)}\lesssim E^{\mathrm{in}} + \nu^{-1}\n{e^{\epsilon  \nu^{\frac13} t}(f_1,f_2)}^2_{L^2(0, T; L^2)},
\end{align*}
which completes the proof of Proposition \ref{lemma: uL2+omL2}.
\end{proof}

\subsection{Space-time estimates for the full problem}\label{sec: proof full pro}

We are in a position to prove Proposition \ref{lemma:est,u,om,timespace,f12}.

\begin{proof}[Proof of Proposition \ref{lemma:est,u,om,timespace,f12}]
 For the case $\nu k^2\geq \n{\pa_{y}U}_{L_{t}^{\infty}L_{y}^{\infty}}$, the desired estimates have already been established in Proposition \ref{Th xi small}.
Now, we focus on the low frequency case $\nu k^2\leq \n{\pa_{y}U}_{L_{t}^{\infty}L_{y}^{\infty}}$.

In Proposition \ref{lemma: uL2+omL2}, we have established the following estimates
\begin{align}\label{esti:omL2L2}
&|k|^2\|e^{\epsilon \nu^{\frac13}t}u\|^2_{L^2 L^2}+\nu^{\frac12} |k|\|e^{\epsilon \nu^{\frac13}t}\om\|^2_{L^2 L^2}+\nu^{\frac13}|k|^{\frac23}\|e^{\epsilon \nu^{\frac13}t}\rho^{\frac12}_k \om\|^2_{L^2L^2}\\
&\lesssim E^{\mathrm{in}}+\nu^{-1}\|e^{\epsilon \nu^{\frac13}t}(f_1,f_2)\|^2_{L^2 L^2}.\notag
\end{align}
It reduces to dealing with the remaining energy terms, which is split into the following three steps.

\textbf{Step 1.} Estimate of $\nu^{\frac14}|k|^{\frac12}\|e^{\epsilon \nu^{\frac13}t} \om\|_{L^\infty L^2}$.\smallskip

Testing \eqref{equ: om} by $\om$ and using integration by parts, we obtain
\begin{align*}
\frac12\frac{d}{dt}\|\om\|^2_{L^2}+\nu \|(\pa_y,k)\om\|^2_{L^2}
= &-\lan ikU \om-ik \pa^2_y U\psi , \om\ran +\lan -ikf_1-\pa_y f_2, \om\ran+\nu \pa_y \om \om \big|_{y=0}^{y=1}\\
= &-ik \int^1_0  U|\om|^2 + \pa^2_y U (|\pa_y \psi|^2+k^2|\psi|^2) dy -ik \int^1_0 \pa^3_y U\psi \pa_y \psi dy \\
&-\int^1_0 ikf_1 \om dy+\int^1_0 f_2\pa_y\om dy+[(\nu \pa_y \om-f_2)\om]\big|_{y=0}^{y=1}.
\end{align*}
Taking the real part and using \eqref{equ:U,H4,estimate}, we deduce
\begin{align}\label{dtomL2+nupaomL2}
&\frac12\frac{d}{dt}\|\om\|^2_{L^2}+\nu \|(\pa_y,k)\om\|^2_{L^2}\\
&\leq  C|k|\|\psi\|_{L^2}\|\pa_y \psi\|_{L^2}+|k|\|f_1\|_{L^2}\|\om\|_{L^2}+\|f_2\|_{L^2}\|\pa_y\om\|_{L^2}+\|\om\|_{L^\infty}\Big|(\nu \pa_y \om-f_2)\big|_{y=0}^{y=1}\Big|\notag\\
&\leq C\|u\|^2_{L^2}+\nu^{-1}\|(f_1,f_2)\|^2_{L^2}+\frac14\nu\|(\pa_y,k)\om\|^2_{L^2}+\|\om\|_{L^\infty}\Big|(\nu \pa_y \om-f_2)\big|_{y=0}^{y=1}\Big|.\notag
\end{align}

For simplicity, we set $\gamma_{1}=\frac{\sinh k(1-y)}{\sinh k}$, which satisfies
$(\pa^2_y -k^2)\gamma_1=0$.
Due to the no-slip boundary condition, we have
\begin{align*}
\lan \pa_t \om,\gamma_{1}\ran =\lan \pa_t \psi, (\pa^2_y-k^2) \gamma_1\ran =0.
\end{align*}
Then we get by \eqref{equ: om} that
\begin{align*}
0=&\lan \pa_t \om, \gamma_{1}\ran=\lan \nu(\pa^2_y-k^2)\om, \gamma_1\ran-\lan ik U\om,\gamma_1\ran+\lan ik\pa^2_y U\psi,\gamma_1\ran -\lan ikf_1,\gamma_1\ran-\lan \pa_yf_2,\gamma\ran\\
=&\lan -ik U\om,\gamma_1\ran +\lan ik\pa^2_y U\psi,\gamma_1\ran-\lan ikf_1,\gamma_1\ran+\lan f_2,\pa_y\gamma_1\ran +[-\nu \om \pa_y \gamma_1 +(\nu \pa_y\om-f_2)\gamma_1](0).
\end{align*}
Noting that
\begin{align*}
\lan -ik U\om,\gamma_1\ran =&\lan ik \pa_y U\pa_y\psi, \gamma_1\ran+\lan ik U\pa_y\psi, \pa_y \gamma_1\ran+\lan ik^3 U\psi, \gamma_1\ran \\
=&-\lan ik\pa^2_y U \psi,\gamma_1\ran -2\lan ik \pa_y U\psi,\pa_y \gamma_1\ran -\lan ik U\psi, (\pa^2_y-k^2)\gamma_1\ran\\
=&-\lan ik \pa^2_y U \psi,\gamma_1\ran-2\lan ik \pa_y U\psi,\pa_y\gamma_1\ran,
\end{align*}
and using \eqref{equ:U,H4,estimate}, Lemma \ref{lemma: sinh}, we deduce
\begin{align}\label{boundary0}
|(\nu \pa_y\om-f_2)(0)|\lesssim \nu |k|\|\om\|_{L^\infty}+|k|^{\frac12}\|(f_1,f_2)\|_{L^2}+|k|^{\frac32}\|\psi\|_{L^2}.
\end{align}

Similarly, we have
\begin{align}\label{boundary1}
|(\nu \pa_y\om-f_2)(1)|\lesssim \nu |k|\|\om\|_{L^\infty}+|k|^{\frac12}\|(f_1,f_2)\|_{L^2}+|k|^{\frac32}\|\psi\|_{L^2}.
\end{align}

Inserting \eqref{boundary0} and \eqref{boundary1} into \eqref{dtomL2+nupaomL2}, we derive
\begin{align}\label{omL2payomL2}
&\frac12\frac{d}{dt}\|\om\|^2_{L^2}+\nu \|(\pa_y,k)\om\|^2_{L^2}\\
\lesssim  & \|u\|^2_{L^2}+\nu^{-1}\|(f_1,f_2)\|^2_{ L^2}+\|\om\|_{L^\infty}(\nu |k|\|\om\|_{L^\infty}+|k|^{\frac12}\|(f_1,f_2)\|_{L^2}+|k|^{\frac32}\|\psi\|_{L^2}).\notag
\end{align}
By the interpolation inequality
$\|\om\|_{L^\infty}\leq \|\om\|^{\frac12}_{L^2}\|(\pa_y,k)\om\|^{\frac12}_{L^2}$,
and Young's inequality, we have
\begin{align*}
\nu |k| \|\om\|^2_{L^\infty}\leq &  \eta \nu \|(\pa_y,k) \om\|^2_{L^2}+C_{\eta}\nu k^2 \|\om\|^2_{L^2},\\
|k|^{\frac12}\|\om\|_{L^\infty}\|(f_1,f_2)\|_{L^2}\leq & \eta \nu\|(\pa_y,k)\om\|^2_{L^2}+C_{\eta}\nu^{-1}\|(f_1,f_2)\|^2_{L^2},\\
|k|^{\frac32}\|\psi\|_{L^2} \|\om\|_{L^\infty}\leq & \eta\nu \|(\pa_y,k)\om\|^2_{L^2}+C_{\eta}(\nu^{-\frac14}|k|^{\frac32} \|\psi\|_{L^2}\|\om\|^{\frac12}_{L^2})^{\frac43}\\
\leq &\eta\nu \|(\pa_y,k)\om\|^2_{L^2}+C_{\eta} \nu^{-\frac12}|k|^{-1}(|k|^2\|u\|^2_{L^2}+\nu^{\frac12}|k|\|\om\|^2_{L^2}).
\end{align*}
Substituting the estimates above into \eqref{omL2payomL2} with $\eta\ll 1$, we arrive at
\begin{align*}
\frac12\frac{d}{dt}\|\om\|^2_{L^2}+\nu \|(\pa_y,k)\om\|^2_{L^2}\lesssim \nu^{-1}\|(f_1,f_2)\|^2_{L^2}+\nu^{-\frac12}|k|^{-1}(|k|^2\|u\|^2_{L^2}+\nu^{\frac12}|k|\|\om\|^2_{L^2}),
\end{align*}
which implies that
\begin{align*}
&\frac{d}{dt}\|e^{\epsilon \nu^{\frac13}t}\om\|^2_{ L^2}+\nu \|e^{\epsilon \nu^{\frac13}t}(\pa_y,k)\om\|^2_{L^2}\\
&\lesssim \nu^{-1}\|e^{\epsilon \nu^{\frac13}t}(f_1,f_2)\|^2_{L^2}+\nu^{-\frac12}|k|^{-1}(|k|^2\|e^{\epsilon \nu^{\frac13}t}u\|^2_{L^2}+\nu^{\frac12}|k|\|e^{\epsilon \nu^{\frac13}t}\om\|^2_{L^2}).
\end{align*}
Thanks to \eqref{esti:omL2L2} and $\nu k^2\lesssim 1$, we deduce
\begin{align}\label{omLinftyL2}
&\nu^{\frac12}|k|( \|e^{\epsilon \nu^{\frac13}t}\om\|^2_{L^\infty L^2}+\nu \|e^{\epsilon \nu^{\frac13}t}(\pa_y,k)\om\|^2_{L^2 L^2})\\
&\lesssim  E^{\mathrm{in}}+ \nu^{-1}\|e^{\epsilon \nu^{\frac13}t}(f_1,f_2)\|^2_{L^2 L^2} +|k|^2\|e^{\epsilon \nu^{\frac13}t}u\|^2_{L^2L^2}+\nu^{\frac12}|k|\|e^{\epsilon \nu^{\frac13}t}\om\|^2_{L^2L^2}\notag\\
&\lesssim   E^{\mathrm{in}}+\nu^{-1}\|e^{\epsilon \nu^{\frac13}t}(f_1,f_2)\|^2_{L^2L^2}.\notag
\end{align}

\textbf{Step 2.} Estimate of $\n{e^{\epsilon \nu^{\frac13}t}\sqrt{1-(2y-1)^2}\om}_{L^{\infty}L^{2}}$.\smallskip

Testing \eqref{equ: om} by $(1-(2y-1)^2)\om$ and taking the real part, we obtain
\begin{align}\label{weight esti}
&\frac12 \frac{d}{dt} \|\sqrt{1-(2y-1)^2} \om\|^2_{L^2}+\nu \|\sqrt{1-(2y-1)^2}(\pa_y,k) \om\|^2_{L^2}\\
&\leq  \nu|\operatorname{Re}\lan \pa_y \om,  2(2y-1) \om\ran| + |\operatorname{Re} \lan ik \pa^2_y U\psi, (1-(2y-1)^2)\om\ran |\notag\\
&\quad+|\lan ik f_1, (1-(2y-1)^2)\om\ran|+|\lan f_2, \pa_y((1-(2y-1)^2)\om)\ran|.\notag
\end{align}
By integration by parts and H\"older's inequality, \eqref{equ:U,H4,estimate}, we obtain
\begin{align}
\nu |\operatorname{Re} \lan \pa_y \om,  2(2y-1) \om\ran|\leq &2\nu \|\pa_y \om\|_{L^2}\|\om\|_{L^2}\lesssim \nu^{\frac32}\|\pa_y\om\|^2_{L^2}+\nu^{\frac12} \|\om\|^2_{L^2}, \label{inner payomom}\\
|\operatorname{Re} \lan ik \pa^2_y U\psi, (1-(2y-1)^2)\om\ran |=&|\operatorname{Re} \lan ik \pa_y(\pa^2_y U(1-(2y-1)^2))\psi,\pa_y\psi\ran |\label{inner phiom}\\
\lesssim &|k|(\|\pa^3_y U\|_{L^\infty}+\|\pa^2_y U\|_{L^\infty})\|\psi\|_{L^2}\|\pa_y\psi\|_{L^2}\lesssim \|u\|^2_{L^2},\notag
\end{align}
and
\begin{align}\label{inner f1om,f2om}
&|\lan ik f_1,(1-(2y-1)^2)\om\ran|+|\lan f_2, \pa_y((1-(2y-1)^2)\om)\ran| \\
&\leq |k|\|f_1\|_{L^2}\|\sqrt{1-(2y-1)^2}\om\|_{L^2}+|\lan f_2, -4y \om+(1-(2y-1)^2)\pa_y\om\ran \notag\\
&\leq C\nu^{-1}\|(f_1, f_2)\|^2_{L^2} +\frac12\nu\|\sqrt{1-(2y-1)^2} (\pa_y, k)\om\|^2_{L^2}+\nu\|\om\|^2_{L^2}.\notag
\end{align}
Inserting \eqref{inner payomom}--\eqref{inner f1om,f2om} into \eqref{weight esti}, we have
\begin{align*}
&\frac{d}{dt} \|\sqrt{1-(2y-1)^2} \om\|^2_{L^2}+\nu \|\sqrt{1-(2y-1)^2}(\pa_y,k) \om\|^2_{L^2}\\
&\lesssim  \nu^{\frac32}\|\pa_y\om\|^2_{L^2}+\nu^{\frac12} \|\om\|^2_{L^2}+\|u\|^2_{L^2}+\nu^{-1}\|(f_1, f_2)\|^2_{L^2},
\end{align*}
which yields that
\begin{align*}
& \|e^{\epsilon \nu^{\frac13}t}\sqrt{1-(2y-1)^2} \om\|^2_{L^\infty L^2}+\nu \|e^{\epsilon \nu^{\frac13}t}\sqrt{1-(2y-1)^2}(\pa_y,k) \om\|^2_{L^2 L^2}\\
&\lesssim \|\om^{\mathrm{in}}_{k}\|^2_{L^2}+\nu^{\frac13}\|e^{\epsilon \nu^{\frac13}t}\sqrt{1-(2y-1)^2}\om\|^2_{L^2L^2}+ \nu^{\frac32}\|e^{\epsilon \nu^{\frac13}t}\pa_y\om\|^2_{L^2L^2}\\
&\quad+\nu^{\frac12} \|e^{\epsilon \nu^{\frac13}t}\om\|^2_{L^{2}L^2}+\|e^{\epsilon \nu^{\frac13}t}u\|^2_{L^2L^2}+\nu^{-1}\|e^{\epsilon \nu^{\frac13}t}(f_1, f_2)\|^2_{L^2L^2}.
\end{align*}

Noticing that $\rho_{k}=1$ for $y\in (\nu^{\frac16} ,1-\nu^{\frac16})$, we have
\begin{align*}
&\nu^{\frac13}\|e^{\epsilon \nu^{\frac13}t}\sqrt{1-(2y-1)^2}\om\|^2_{L^2(0,1)}\\
&\leq \nu^{\frac13}\|e^{\epsilon \nu^{\frac13}t}\rho^{\frac12}_k\om\|^2_{L^2(\nu^{\frac16}, 1-\nu^{\frac16})}+\nu^{\frac13}\|1-(2y-1)^2\|_{L^{\infty}((0,\nu^{\frac16})\cup(1-\nu^{\frac16},1))}\|e^{\epsilon \nu^{\frac13}t}\om\|^2_{L^2}\\
&\lesssim \nu^{\frac13}\|e^{\epsilon \nu^{\frac13}t}\rho^{\frac12}_k\om\|^2_{L^2}+\nu^{\frac12}\|e^{\epsilon \nu^{\frac13}t}\om\|^2_{L^2}.
\end{align*}
Then, using \eqref{esti:omL2L2} and \eqref{omLinftyL2}, we arrive at
\begin{align}\label{weightomL2}
\|e^{\epsilon \nu^{\frac13}t}\sqrt{1-(2y-1)^2} \om\|^2_{L^\infty L^2}
\lesssim   E^{\mathrm{in}}+\nu^{-1}\|e^{\epsilon \nu^{\frac13}t}(f_1, f_2)\|^2_{L^2L^2}.
\end{align}

\textbf{Step 3.} Estimate of $\n{e^{\epsilon \nu^{\frac13}t}u}_{L^{\infty} L^{\infty}}$.\smallskip

Let us introduce the cutoff function $\rho$ defined by
\begin{align*}
\rho(y)=
\left\{
\begin{aligned}
&1, \qquad \qquad\qquad\qquad\,\  y\in (\nu^{\frac14}, 1-\nu^{\frac14}),\\
&\frac12\nu^{-\frac14}(1-|2y-1|),\quad  y\in (0,\nu^{\frac{1}{4}})\cup(1-\nu^{\frac{1}{4}},1).
\end{aligned}
\right.
\end{align*}
Then we have
\begin{align*}
\|u\|_{ L^\infty}\leq \|\om\|_{L^1}=&\int^{\nu^{\frac12}}_0|\om|dy+ \int^{1-\nu^{\frac12}}_{\nu^{\frac12}}|\om|dy+\int^{1}_{1-\nu^{\frac12}}|\om|dy\\
\leq &2\nu^{\frac14}\|\om\|_{L^2}+\|\rho^{-1}\|_{L^2(\nu^{\frac12}, 1-\nu^{\frac12})}\|\rho \om\|_{L^2} .
\end{align*}
Noticing that
\begin{align*}
\|\rho^{-1}\|^2_{L^2(\nu^{\frac12}, 1-\nu^{\frac12})}\lesssim &\int^{1-\nu^{\frac14}}_{\nu^{\frac14}} 1 dy+\int^{1-\nu^{\frac12}}_{1-\nu^{\frac14}} \nu^{\frac12}(1-|2y-1|)^{-2} dy+\int^{\nu^{\frac14}}_{\nu^{\frac12}} \nu^{\frac12}(1-|2y-1|)^{-2} dy\\
\lesssim &1-2\nu^{\frac14}+\nu^{\frac12}(\nu^{-\frac12}-\nu^{-\frac14})\lesssim 1,
\end{align*}
we get by \eqref{omLinftyL2} that
\begin{align}\label{esti:uLinftyLinfty}
\|e^{\epsilon \nu^{\frac13}t }u\|^2_{L^\infty L^\infty}\leq &\nu^{\frac12}\|e^{\epsilon \nu^{\frac13}t }\om\|^2_{L^\infty L^2}+\|e^{\epsilon \nu^{\frac13}t }\rho \om\|^2_{L^\infty L^2}\\
\lesssim  &  E^{\mathrm{in}}+\nu^{-1}\|e^{\epsilon \nu^{\frac13}t}(f_1, f_2)\|^2_{L^2L^2}+\|e^{\epsilon \nu^{\frac13}t }\rho \om\|^2_{L^\infty L^2}. \notag
\end{align}

We are left to estimate $\|e^{\epsilon \nu^{\frac13}t }\rho \om\|_{L^\infty L^2}$. To do this,
we introduce
a $C^{2}$-smooth function
\begin{align*}
\widetilde{\rho}(y)=
\left\{
\begin{aligned}
&1, \qquad \qquad\qquad\qquad\qquad\qquad\quad\,\ y\in (\nu^{\frac14}, 1-\nu^{\frac14}),\\
&1+\Big(\frac12\nu^{-\frac14}(1-|2y-1|)-1\Big)^3, \quad y\in (0,\nu^{\frac{1}{4}})\cup(1-\nu^{\frac{1}{4}},1),
\end{aligned}
\right.
\end{align*}
which satisfies the following properties
\begin{align}\label{esti: rho}
\rho\leq \widetilde{\rho},\qquad |\pa_y \widetilde{\rho}|\lesssim \nu^{-\frac14},\qquad \|\pa_y\widetilde{\rho}\|_{L^2}\lesssim \nu^{-\frac18}, \qquad |\pa^2_y \widetilde{\rho}|\lesssim \nu^{-\frac12}.
\end{align}

Testing \eqref{equ: om} by $\widetilde{\rho}^2\om$ and taking the real part, we have
\begin{align}\label{rhoomL2}
&\frac12\frac{d}{dt}\|\widetilde{\rho}\om\|^2_{L^2}+\nu\|\widetilde{\rho}(\pa_y,k)\om\|^2_{L^2}\\
&\leq |\operatorname{Re}(\nu\lan \pa_y\om, \pa_y(\widetilde{\rho}^2) \om\ran+\lan ik \pa^2_y U\psi, \widetilde{\rho}^2\om\ran) |+|\lan ik f_1, \widetilde{\rho}^2\om\ran|+|\lan \pa_y f_2, \widetilde{\rho}^2\om \ran|.\notag
\end{align}

Using the integration by parts, \eqref{esti: rho} and \eqref{equ:U,H4,estimate}, we obtain
\begin{align}
&\nu|\lan \pa_y\om,\pa_y(\widetilde{\rho}^2) \om\ran|\leq \nu\Big|\int^1_0 \pa^2_y(\widetilde{\rho}^2)|\om|^2dy \Big|\lesssim \nu^{\frac12}\|\om\|^2_{L^2},\label{inner:payomrhoom}
\end{align}
and
\begin{align}\label{inner:psiom}
&|\operatorname{Re}\lan ik \pa^2_yU\psi, \widetilde{\rho}^2\om\ran|\leq |\lan k \pa_y(\pa^2_yU \widetilde{\rho}^2)\psi, \pa_y \psi\ran |\\
&\leq \|\pa^3_yU \widetilde{\rho}^2\|_{L^\infty} \|k\psi\|_{L^2}\|\pa_y \psi\|_{L^2}+2\|\pa^2_yU\widetilde{\rho}\pa_y\widetilde{\rho} \|_{L^2}|k|\|\psi\|_{L^\infty((0,\nu^{\frac{1}{4}})\cup(1-\nu^{\frac{1}{4}},1))}
\|\pa_y\psi\|_{L^2}\notag\\
&\lesssim \n{u}_{L^{2}}^{2}+\nu^{-\frac{1}{8}}|k|\|\psi\|_{L^\infty((0,\nu^{\frac{1}{4}})\cup(1-\nu^{\frac{1}{4}},1))}\|u\|_{L^2}\lesssim |k|\n{u}_{L^{2}}^{2},\notag
\end{align}
where in the last inequality we have used that
\begin{align*}
\|\psi\|_{L^\infty((0,\nu^{\frac{1}{4}})\cup(1-\nu^{\frac{1}{4}},1))}=\bbn{\int^{y}_{0}\pa_y\psi(z)dz}_{L^\infty(0,\nu^{\frac14})}+ \bbn{\int^{y}_{1}\pa_y\psi(z)dz}_{L^\infty(1-\nu^{\frac14},1)}\lesssim \nu^{\frac{1}{8}}\n{\pa_{y}\psi}_{L^{2}}.
\end{align*}

For $|\lan ik f_1, \widetilde{\rho}\om\ran|+|\lan \pa_y f_2, \widetilde{\rho}\om \ran|$, we get by using Cauchy-Schwarz inequality and \eqref{esti: rho} that
\begin{align}\label{inner: f1rhoom+f2rhoom}
&|\lan ik f_1, \widetilde{\rho}\om\ran|+|\lan \pa_y f_2, \widetilde{\rho}\om \ran|\\
&\leq \nu^{-1}\|f_1\|_{L^2}+\frac14 \nu \|k \widetilde{\rho}\om\|^2_{L^2}+\nu^{-1}\|f_2\|^2_{L^2}+\frac14 \nu \|\widetilde{\rho}\pa_y \om\|^2_{L^2}+\frac14\nu \|\pa_y\widetilde{\rho}\om\|^2_{L^2}\notag\\
&\leq  \nu^{-1}\|(f_1,f_2)\|^2_{L^2}+\frac14\nu \|\widetilde{\rho}(\pa_y,k)\om\|^2_{L^2}+\nu^{\frac12}\|\om\|^2_{L^2}.\notag
\end{align}

Combining \eqref{inner:payomrhoom}, \eqref{inner:psiom} and \eqref{inner: f1rhoom+f2rhoom} with \eqref{rhoomL2}, we deduce
\begin{align*}
\frac12\frac{d}{dt}\|\widetilde{\rho}\om\|^2_{L^2}+\nu\|\widetilde{\rho}(\pa_y,k)\om\|^2_{L^2}
\lesssim &\nu^{-1}\|(f_1,f_2)\|^2_{L^2} +\nu^{\frac12}\|\om\|^2_{L^2}+|k|^2\|u\|^2_{L^2}.
\end{align*}
Recalling the definition \eqref{equ:def,rho,k} of the weight function $\rho_{k}$, by a direct calculation, we have
$\widetilde{\rho}\leq \rho^{\frac12}_{k} $. Then, we use the space-time estimate \eqref{esti:omL2L2} to obtain
\begin{align*}
&\|e^{\epsilon \nu^{\frac13}t}\widetilde{\rho}\om\|^2_{L^\infty L^2}+\nu\|e^{\epsilon \nu^{\frac13}t}\widetilde{\rho}(\pa_y,k)\om\|^2_{L^2 L^2}\\
&\lesssim  E^{\mathrm{in}}+\nu^{\frac13}\|e^{\epsilon \nu^{\frac13}t}\rho^{\frac12}_{k}\om\|^2_{L^2L^2} +\nu^{\frac12}\|e^{\epsilon \nu^{\frac13}t}\om\|^2_{L^2L^2}+|k|^2\|e^{\epsilon \nu^{\frac13}t}u\|^2_{L^2L^2}+\nu^{-1}\|e^{\epsilon \nu^{\frac13}t}(f_1,f_2)\|^2_{L^2L^2}\\
&\lesssim  E^{\mathrm{in}}+\nu^{-1}\|e^{\epsilon \nu^{\frac13}t}(f_1,f_2)\|^2_{L^2L^2}.
\end{align*}
Inserting the estimate above into \eqref{esti:uLinftyLinfty}, with $\rho \lesssim \wt{\rho}$, we arrive at
\begin{align}\label{uLinftyLinfty}
\|e^{\epsilon \nu^{\frac13}t}u\|^2_{L^\infty L^\infty}\lesssim E^{\mathrm{in}}+\nu^{-1}\|e^{\epsilon \nu^{\frac13}t}(f_1,f_2)\|^2_{L^2L^2}.
\end{align}

This completes the proof of Proposition \ref{lemma:est,u,om,timespace,f12}.
\end{proof}

\section{Nonlinear stability}\label{nonlinear tran thre}
This section is devoted to deriving the asymptotic stability threshold of the monotone shear flow with no-slip boundary condition, using the space-time estimates in Section \ref{sec: space-time estimate}.

Let us recall the stability norms
\begin{equation*}
E_k=\left\{
\begin{aligned}
&\|\om_0\|_{L^\infty L^2}, \quad k=0,\\
&|k|\| e^{\epsilon_0 \nu^{\frac13}t} u_k\|_{L^2L^2}+\|e^{\epsilon_0 \nu^{\frac13}t} u_k\|_{L^\infty L^\infty}\\
&+\|e^{\epsilon_0 \nu^{\frac13}t} \sqrt{1-(2y-1)^2} \om_k\|^2_{L^\infty L^2}+\nu^{\frac14}|k|^{\frac12}\|e^{\epsilon_0 \nu^{\frac13}t} \om_k\|_{L^2L^2}, \quad k\neq 0.
\end{aligned}
\right.
\end{equation*}
\begin{proof}[Proof of Theorem \ref{Th: tran thre}]
The equation \eqref{equ: omli11} can be rewritten as
\begin{align*}
\pat\om-\nu\De \om+U\pa_x\om-\pa^2_y U\pa_x\psi=-\operatorname{curl}(u\cdot\na u).
\end{align*}
Taking the Fourier transform with respect to $x$-variable, we obtain the $k$--frequency formulation of the equation:
\begin{align*}
&\pa_t \om_k-\nu (\pa^2_y-k^2) \om_k+ikU \om_k-ik\pa^2_y U\psi_k\\
&\quad=-\pa_y(u\cdot\na u^{(1)})_k+ik(u\cdot\na u^{(2)})_k
:=-\pa_y (f^{1,1}_k+f^{2,1}_k)+ik(f^{1,2}_k+f^{2,2}_k),
\end{align*}
where
\begin{align*}
&f^{1,1}_k=i\sum_{l\in\Z} u^{(1)}_l(t,y)(k-l)u^{(1)}_{k-l}(t,y), \quad f^{1,2}_k=i\sum_{l\in\Z} u^{(1)}_l(t,y)(k-l)u^{(2)}_{k-l}(t,y),\\
&f^{2,1}_k=\sum_{l\in \Z} u^{(2)}_l(t,y) \pa_y u^{(1)}_{k-l}(t,y), \qquad\quad f^{2,2}_k=\sum_{l\in \Z} u^{(2)}_l(t,y) \pa_y u^{(2)}_{k-l}(t,y).
\end{align*}

We now analyze the energy term $E_0$. The condition $\operatorname{div} u=0$ implies $u^{(2)}_0=0$.
Combined with $P_0(u^{(1)}\pa_x u^{(1)})=0,$ the zero-frequency mode of \eqref{pertu} for $u^{1}$ yields that
\begin{align}\label{new,equ:u01,def}
(\pa_t -\nu \pa^2_y) u^{(1)}_0(t,y)=-\sum_{l\in \Z\setminus \{0\}} u^{(2)}_l \pa_y u^{(1)}_{-l}(t,y):=-f_0^{2,1}(t,y).
\end{align}
Testing equation \eqref{new,equ:u01,def} by $-\pa_y^2u_0^{(1)}$ and using the integration by parts, we obtain
\begin{align*}
\pa_t \|\pa_y u^{(1)}_0\|^2_{L^2}+2\nu \|\pa^2_y u^{(1)}_0\|^2_{L^2}=2\lan f_0^{2,1}, \pa^2_y u^{(1)}_0\ran \leq C \nu^{-1}\|f_0^{2,1}\|^2_{L^2}+\frac12\nu\|\pa^2_y u^{(1)}_0\|^2_{L^2}.
\end{align*}
which, together with $\pa_yu_0^{(1)}(t,y)=\om_0(t,y)$, gives
\begin{align}\label{esti: E0}
E^2_0\lesssim \nu^{-1}\|f_0^{2,1}\|^2_{L^2L^2}+\|\om^{\mathrm{in}}_0\|^2_{L^2}.
\end{align}

For the energy term $E_k$ with $k\neq 0$, it follows from Proposition \ref{lemma:est,u,om,timespace,f12} that
\begin{equation}\label{esti: Ek}
\begin{aligned}
E^2_{k}\lesssim \nu^{-1} \|e^{\epsilon_0 \nu^{\frac13}t} f^{i,j}_k\|^2_{L^2L^2} +\|u^{\mathrm{in}}_k\|^2_{H^1}+|k|^{-2}\|\pa_y\om^{\mathrm{in}}_{k}\|^2_{L^2}, \quad 1\leq i,j\leq 2.
\end{aligned}
\end{equation}

For $f^{1,i}$ with $i=1,2$, by H\"older's inequality, we have
\begin{align}\label{esti: f1}
\|e^{\epsilon_0 \nu^{\frac13}t}f^{1,i}\|_{L^2L^2}\leq \sum_{l\in\Z}\|e^{\epsilon_0 \nu^{\frac13}t}u^{(1)}_l\|_{L^\infty L^\infty}\|(k-l)e^{\epsilon_0 \nu^{\frac13}t}u^{(i)}_{k-l}\|_{L^2L^2}\leq \sum_{l\in \Z} E_l E_{k-l}.
\end{align}

For $f^{2,i}$ with $i=1,2$, we first use Hardy's inequality and the fact that $\pa_{y}u^{(2)}=-iku^{(1)}$ to obtain
\begin{align*}
\Big\|\frac{e^{\epsilon_0 \nu^{\frac13}t}u^{(2)}_{l}}{\sqrt{1-(2y-1)^2}}\Big\|_{L^2 L^\infty}
\lesssim \|e^{\epsilon_0 \nu^{\frac13}t}\pa_y u^{(2)}_l\|_{L^2L^2}
\lesssim \|e^{\epsilon_0 \nu^{\frac13}t}lu^{(1)}_l\|_{L^2L^2}\lesssim E_{l}.
\end{align*}
This along with \eqref{esti: weightpayu} implies that
\begin{align}\label{esti: f2}
\|e^{\epsilon_0 \nu^{\frac13}t}f^{2,i}\|_{L^2L^2}
\leq & \sum_{l\in\Z} \Big\|\frac{e^{\epsilon_0 \nu^{\frac13}t}u^{(2)}_{l}}{\sqrt{1-(2y-1)^2}}\Big\|_{L^2 L^\infty}\|e^{\epsilon_0 \nu^{\frac13}t}\sqrt{1-(2y-1)^2} \pa_y u^{(i)}_{k-l}\|_{L^\infty L^2}\\
\lesssim & \sum_{l\in\Z} E_{l}\|e^{\epsilon_0 \nu^{\frac13}t}\sqrt{1-(2y-1)^2}\om_{k-l}\|_{L^\infty L^2}\lesssim \sum_{l\in\Z}E_{l}E_{k-l}.\notag
\end{align}

Inserting \eqref{esti: f1}, \eqref{esti: f2} into \eqref{esti: Ek} and \eqref{esti: f2} into \eqref{esti: E0}, we deduce
\begin{align*}
E_k\lesssim \nu^{-\frac12}\sum_{l\in\Z} E_{l}E_{k-l}+\|u^{\mathrm{in}}_k\|_{H^1}+|k|^{-1}\|\pa_y\om^{\mathrm{in}}_{k}\|_{L^2}, \quad k\in\Z.
\end{align*}
Summing up with $k$, we further obtain
\begin{align*}
\sum_{k\in\Z} E_k\lesssim \nu^{-\frac12}\sum_{k\in\Z}\sum_{l\in \Z} E_{l}E_{k-l} +\sum_{k\in\Z}(\|u^{\mathrm{in}}_k\|_{H^1}+|k|^{-1}\|\pa_y\om^{\mathrm{in}}_{k}\|_{L^2}).
\end{align*}
Thanks to the initial data condition
\begin{align*}
\|u^{\mathrm{in}}\|_{H^2}\leq c\nu^{\frac12},
\end{align*}
we finally derive
\begin{align*}
\sum_{k\in\Z}E_k\lesssim \nu^{\frac12},
\end{align*}
which completes the proof of Theorem \ref{Th: tran thre}.
\end{proof}

\appendix
\section{Auxiliary estimates}

\begin{lemma}\label{lemma: sinh}
There holds that
\begin{align}
&\bbn{\frac{\sinh{k(1-y)}}{\sinh k}}_{L^2}\lesssim |k|^{-\frac12}, \qquad \bbn{\frac{\sinh{(ky)}}{\sinh k}}_{L^2}\lesssim |k|^{-\frac12}, \label{esti:sinh}\\
&\bbn{\frac{\cosh{k(1-y)}}{\sinh k}}_{L^2}\lesssim |k|^{-\frac12}, \qquad \bbn{\frac{\cosh{(ky)}}{\sinh k}}_{L^2}\lesssim |k|^{-\frac12}. \label{esti:cosh}
\end{align}
\end{lemma}
\begin{lemma}
Let $(\pa^2_y- k^2)\psi=\om$ with $\psi(\pm1)=0$. Then we have
\begin{align}\label{esti: weightpayu}
\|\sqrt{1-(2y-1)^2} \pa_y u\|^2_{L^2}\leq \|\sqrt{1-(2y-1)^2}\om\|^2_{L^2}.
\end{align}
\end{lemma}
\begin{proof}
By the definition of $\om$, we have
\begin{align*}
&\|\sqrt{1-(2y-1)^2}\om\|^2_{L^2}=\int^{1}_{0}(1-(2y-1)^2)|\om|^2dy\\
=&\int^{1}_{0} (1-(2y-1)^2)(|\pa^2_y\psi|^2+|k^2\psi|^2)dy-2\re\int^{1}_{0}(1-(2y-1)^2) \pa^2_y\psi |k|^2\psi dy\\
=&\int^{1}_{0} (1-(2y-1)^2)(|\pa^2_y\psi|^2+|k^2\psi|^2+2|k\pa_y\psi|^2)dy-4\re\int^{1}_{0}(2 y-1) \pa_y \psi |k|^2\psi dy\\
=& \int^{1}_{0} (1-(2y-1)^2)(|\pa^2_y\psi|^2+|k^2\psi|^2+2|k\pa_y\psi|^2)dy+4\int^{1}_{0}|k\psi|^2dy,
\end{align*}
which implies that
\begin{align*}
\|\sqrt{1-(2y-1)^2} \pa_y u\|^2_{L^2}\leq \|\sqrt{1-(2y-1)^2}\om\|^2_{L^2}.
\end{align*}
\end{proof}

\begin{lemma}\label{lemma:heat,dirichlet,H2}
Let $f(t)$ be the solution to the 1D heat equation with the Dirichlet boundary condition
\begin{equation}
\left\{
\begin{aligned}
&\pa_{t}f=\nu \pa_{y}^{2}f,\\
&f(t,0)=f(t,1)=0,\quad f(0,y)=f^{\mathrm{in}}(y)\in H^{2}\cap H_{0}^{1}(0,1).
\end{aligned}
\right.
\end{equation}
Then we have
\begin{align}\label{equ:heat,dirichlet,H2}
\n{f}_{L_{t}^{\infty}H^{k}}\leq \n{f^{\mathrm{in}}}_{H^{k}},\quad k=0,1,2.
\end{align}
\end{lemma}
\begin{proof}
For $k=0,1$, the lemma follows from the standard energy method. For $k=2$, we employ an approximation argument.
Let $P_{\leq N}$ denote the spectral projection onto the first $N$ eigenfunctions of the Dirichlet Laplacian, and define $f_N = P_{\leq N} f$.
Applying the energy method to $f_N$, we obtain
\begin{align}
\n{\pa_{y}^{2}f_{N}}_{L_{t}^{\infty}L^{2}}\leq \n{\pa_{y}^{2}f_{N}^{\mathrm{in}}}_{L^{2}}\leq
\n{\pa_{y}^{2}f^{\mathrm{in}}}_{L^{2}}.
\end{align}
Passing to the limit, we deduce the same bound for $f$.
\end{proof}

\section*{Acknowledgments}
Z. Li was partially supported by the Postdoctoral Fellowship Program of CPSF (Grant No: GZC20240123 and 2025M773073). S. Shen is partially supported by NSF of China under Grant 12501322 and Anhui Provincial NSF 2508085QA001.
Z. Zhang is partially supported by NSF of China under Grant 12288101.

\bigskip

\medskip
\noindent\textbf{Data Availability Statement:}
Data sharing is not applicable to this article as no datasets were generated or analysed during the current study.

\noindent\textbf{Conflict of Interest:}
The authors declare that they have no conflict of interest.

\end{document}